\documentclass[a4paper,12pt,reqno]{amsart}

\usepackage[normalem]{ulem}

\usepackage{fullpage}
\usepackage{amssymb}
\usepackage{amsmath}
\usepackage{mathtools}
\usepackage{enumitem}
\usepackage{mathrsfs}
\usepackage[all]{xy}
\usepackage{mathtools}
\usepackage{hyperref}
\usepackage{url}
\usepackage{bbm}
\usepackage{longtable}
\usepackage{dsfont}
\usepackage{upgreek}
\usepackage{lmodern}
\usepackage[OT2, T1]{fontenc}

\DeclareSymbolFont{cyrletters}{OT2}{wncyr}{m}{n}
\usepackage[british]{babel}

\usepackage[margin=0.75in]{geometry}
\numberwithin{equation}{section} \numberwithin{figure}{section}

\let\Im\relax
\DeclareMathOperator{\Im}{Im}

\DeclareSymbolFont{cyrletters}{OT2}{wncyr}{m}{n}
\DeclareMathSymbol{\Sha}{\mathalpha}{cyrletters}{"58}
\DeclareMathSymbol{\Be}{\mathalpha}{cyrletters}{"42}

\renewcommand\P{\mathbb{P}}
\newcommand\Z{\mathbb{Z}}
\newcommand\N{\mathbb{N}}
\newcommand\Q{\mathbb{Q}}
\newcommand\R{\mathbb{R}}
\newcommand\C{\mathbb{C}}

\newcommand{\m}{\mathfrak{m}}

\renewcommand{\b}{\mathbf}

\renewcommand{\leq}{\leqslant}
\renewcommand{\geq}{\geqslant}
\renewcommand{\#}{\sharp}

\newcommand{\p}{\mathfrak{p}}

\newtheorem{lemma}{Lemma}

\newtheorem{theorem}[lemma]{Theorem}
\newtheorem{proposition}[lemma]{Proposition}
\newtheorem{corollary}[lemma]{Corollary}

\theoremstyle{definition}

\newtheorem{remark}[lemma]{Remark}

\usepackage[usenames,dvipsnames]{color}

\numberwithin{lemma}{section}

\title{An improved large sieve for quadratic characters via Hooley neutralisers and its applications}

\author{Cameron Wilson} 
\address{Department of Mathematics \\
University of Glasgow \\ G12~8QQ United Kingdom}
\email{c.wilson.6@research.gla.ac.uk}

\date{\today}

\begin{document}

\begin{abstract}
    We combine Hooley neutralisers and the large sieve for quadratic characters. We give applications to character sums with a hyperbolic height condition.
\end{abstract}

\maketitle

\setcounter{tocdepth}{1}
\tableofcontents

\section{Introduction}

\subsection{An Improved Large Sieve For Quadratic Characters}
The large sieve for quadratic characters has played a vital role in many areas of mathematics in recent decades. It has seen applications in the theory of $L$-functions \cite{SoundLfunctions}, arithmetic statistics \cite{FK,Heath-BrownSelmerRank}, Manin's conjecture \cite{BH-BQuadricBundle}, and local solubility problems \cite{F--I,KPSS,LRS}. The following bounds, or variations thereof, are the most common versions used in these applications. Suppose $a_n$ and $b_m$ are arbitrary complex sequences bounded in magnitude by $1$ and supported on the odd integers. Then for $N,M\geq 2$, 
\begin{itemize}
    \item Elliott \cite{Elliott,Heath-Brown}: \begin{equation}\label{ElliottIntro}\mathop{\sum\sum}_{n\leq N,\;m\leq M}a_nb_m\left(\frac{n}{m}\right)\ll NM\left(N^{-1/2}+N^{1/2}M^{-1/2}\log N\right);\end{equation}
    \item Heath-Brown \cite{Heath-Brown}: for any $\epsilon>0$, \begin{equation}\label{IntroHBbound}\mathop{\sum\sum}_{n\leq N,\;m\leq M}\mu^2(2n)\mu^2(2m)a_nb_m\left(\frac{n}{m}\right)\ll_{\epsilon} (NM)^{1+\epsilon}\left(N^{-1/2}+M^{-1/2}\right);\end{equation}
    \item Friedlander--Iwaniec \cite{F--I}: \begin{equation}\label{IntroF--I}\mathop{\sum\sum}_{n\leq N,\;m\leq M}a_nb_m\left(\frac{n}{m}\right) \ll NM\left(N^{-1/6}+M^{-1/6}\right)\left(\log 3NM\right)^{7/6}.\end{equation}
\end{itemize}
When $N^2\ll M$ the Elliott bound is the most effective bound; when $N$ and $M$ are of comparable size the Heath-Brown result is more effective. In applications to rational points problems the Friedlander--Iwaniec bound is more versatile as the presence of $(NM)^{\epsilon}$ in the Heath-Brown bound may lead to problems when we require no loss of logarithms. One may also combine the Elliott and Heath-Brown results into one as in the work of Fouvry and Kl\"uners on the $4$-rank of class groups \cite[Lemma 15]{FK}. The primary purpose of the present paper is to present an improved version of Elliott's bound whenever the sequence $b_m$ has some multiplicative structure:

\begin{theorem}\label{Hooleyandthelargesieve1}
    Let $M,N\geq 2$, and fix some $\epsilon>0$. Let $f$ be any multiplicative function such that $0\leq f(p)\leq 1$ and $f(p^m)\leq f(p)$ for all primes $p$ and all $m\geq 2$. Suppose also that there exists an $0<\alpha\leq 1$ such that for all $X\geq 2$ we have,
    \begin{equation}\label{Mertenassumption}
    \sum_{p\leq X}\frac{f(p)}{p} = \alpha\log\log X + O(1).
    \end{equation}
    Then for any complex sequences $a_n$,$b_m$ which are supported on the odd integers such that $\lvert a_n\rvert\leq 1$ and $\lvert b_m\rvert\leq 1$ we have:
    \[
    \mathop{\sum\sum}_{n\leq N,\;m\leq M} a_n b_m f(m)\left(\frac{n}{m}\right) \ll_{\epsilon} \frac{MN^{1/2}(\log N)}{(\log M)^{(1-\alpha)}} + \frac{M^{1/2+\epsilon}N^{3/2}(\log N)^{1/2}}{(\log M)^{(1-\alpha)/2}}
    \]
    where the implied constant depends at most on $\epsilon$.
\end{theorem}

This first term in this result dominates when $\frac{N^2}{\log N}\leq \frac{M^{1-\epsilon}}{(\log M)^{1-\alpha}}$ for any $\epsilon>0$. In this range, Theorem \ref{Hooleyandthelargesieve1} improves upon all of the bounds \eqref{ElliottIntro}-\eqref{IntroF--I}. The main benefit here is that we have maintained savings from summing over the multiplicative function. Indeed, it follows from a result of Shiu \cite[Theorem 1]{Shiu} that for a multiplicative function satisfying the conditions of Theorem \ref{Hooleyandthelargesieve1} that we have 
\begin{equation}\label{ShiuIntro}
\sum_{m\leq M} f(m) \ll \frac{M}{(\log M)^{1-\alpha}}.
\end{equation}
We will prove this in \S\ref{HooleyLargeSieve} by using Hooley Neutralisers to insert the Brun Sieve into standard large sieve methods. Using partial summation, we obtain a further improvement when $a_n$ contains a harmonic factor:
\begin{corollary}\label{Hooleyandthelargesieve2}
    Let $M\geq N\geq W\geq 2$, and fix some $\epsilon>0$. Let $f$ be any multiplicative function such that $0\leq f(p)\leq 1$ and $f(p^m)\leq f(p)$ for all primes $p$ and all $m\geq 2$. Suppose also that for all $X\geq 2$ we have,
    \begin{equation}
    \sum_{p\leq X}\frac{f(p)}{p} = \alpha\log\log X + O(1)
    \end{equation}
    for some $0<\alpha<1$. Then for any complex sequences $a_n$,$b_m$ which are supported on the odd integers such that $\lvert a_n\rvert\leq 1$ and $\lvert b_m\rvert\leq1$ we have:
    \[
    \mathop{\sum\sum}_{W<n\leq N,\;m\leq M} \frac{a_n}{n} b_mf(m)\left(\frac{n}{m}\right) \ll_{\epsilon} \frac{M(\log N)}{W^{1/2}(\log M)^{(1-\alpha)}} + \frac{M^{1/2+\epsilon}N^{1/2}(\log N)^{1/2}}{(\log M)^{(1-\alpha)/2}},
    \]
    where the implied constant depends at most on $\epsilon$.
\end{corollary}
The benefit of this result is that it encodes not only the saving from the multiplicative function, but also the convergence from the sum $\sum_{W<n\leq N}\frac{1}{n}\left(\frac{n}{m}\right)$ when $m$ is non-square. Again, this result is most effective when $N^2\ll M$. The utility of this bound in applications is found in applying it for such hyper-skewed regions and applying the Heath-Brown bound \eqref{IntroHBbound} or the Friedlander--Iwaniec bound \eqref{IntroF--I} in regions where $M$ or $N$ are more comparable. To demonstrate this, we provide the following two corollaries. The first is on the simultaneous average $\frac{1}{\tau(m)}$ and special values of $L$-functions, $L\left(1,\left(\frac{\cdot}{m}\right)\right)$.
\begin{corollary}\label{LfunctionweirdaverageIntro}
    For all $X\geq 3$ we have
    \[
    \mathop{\sum}_{\substack{1<m\leq X}}\frac{\mu^2(2m)}{\tau(m)}L\left(1,\left(\frac{\cdot}{m}\right)\right) \ll \frac{X}{\sqrt{\log X}}.
    \]
\end{corollary}
This will be proven as a direct consequence of the more general Lemma \ref{Lfunctionweirdaverage}, whose generality is required for applications to character sums later in the paper. Since the average of $\frac{1}{\tau(m)}$ is $\frac{c_1}{\sqrt{\log M}}$ for some constant $c_1$ and the average of $L\left(1,\left(\frac{\cdot}{m}\right)\right)$ is a constant, this bound suggests that the distribution of these two functions are independent. The second application is of a similar nature.
\begin{corollary}\label{AveragingovermediumconductorsIntro}
    Suppose $a_n$,$b_m$ are any complex sequences supported on odd integers such that $\lvert a_n\rvert\leq 1$ and $\lvert b_m\rvert\leq 1$. Then for any $X\geq 3$, $C_1,C_2>1$ such that $(C_1\log\log X)^{C_2}>2$, and any fixed integers $1\leq c_0,c_1\leq (\log X)^{C_1/32}$ we have
    \[
    \sum_{\substack{(\log X)^{3C_1/4}<n_0c_0\leq X^{1/2}}}\frac{1}{n_0^2c_0^2\tau(n_0)}\left\lvert\mathop{\sum\sum}_{\substack{(C_1\log \log X)^{C_2}< m \leq (\log X)^{2C_1}\\ n_1c_1\leq n_0c_0}}\frac{a_m b_{n_1}}{m\tau(n_1)}\left(\frac{n_1}{m}\right)\right\rvert \ll_{C_1,C_2} \frac{1}{c_0c_1(\log \log X)^{C_3}}
    \]
    where $C_3 = C_2/2-1$ and the implied constant depends at most on $C_1$ and $C_2$.
\end{corollary}
Note that this statement is stated in a way that is convenient when dealing with certain character sums that occur in \S\ref{hyperbolicanalysis1}. On its own, however, this result nicely demonstrates the effectiveness of Theorem \ref{Hooleyandthelargesieve1}. Let $T(X)$ denote the sum on the left-hand side of Corollary \ref{AveragingovermediumconductorsIntro} when $c_0=c_1=1$. Then, using the triangle inequality (i.e, ignoring the oscillation of the Jacobi symbol) and bounding the sum over $n_1$ using \eqref{ShiuIntro} ($\alpha=1/2$), we obtain
\begin{equation}\label{mediumconductortrivialbound}
T(X) \ll_{C_1} \sum_{\substack{(\log X)^{3C_1/4}<n_0\leq X^{1/2}}}\frac{\log\log X}{n_0\tau(n_0)\sqrt{\log n_0}} \ll (\log\log X)^2.
\end{equation}
This second bound is obtained using partial summations and \eqref{ShiuIntro}. On the other hand, applying partial summation to the sum over $m$ in $T(X)$ and then using only \eqref{IntroF--I} in the resulting sums over $n_1$ and $m$ (i.e ignoring the $\frac{1}{\tau}$ factor), the inner sums of $T(X)$ are bounded by
\begin{equation}\label{standardF--Ibound}
\mathop{\sum\sum}_{\substack{(C_1\log \log X)^{C_2}< m \leq (\log X)^{2C_1}\\ n_1c_1\leq n_0c_0}}\frac{a_m b_{n_1}}{m\tau(n_1)}\left(\frac{n_1}{m}\right) \ll_{C_1,C_2} \frac{n_0(\log n_0)^{7/6}}{(\log\log X)^{A}}
\end{equation}
for some $A>0$ depending on $C_2$. Inserting this into $T(X)$ and then summing over $n_0$ using partial summation and \eqref{ShiuIntro} we have
\[
T(X) \ll_{C_1,C_2} \frac{\sqrt{\log X}}{(\log\log X)^{A'}}
\]
for some other $A'>0$. In the applications discussed below, we will require that expressions of the shape $T(X)$ are at least $o(\log\log X)$ and so neither of the bounds given above will be admissible. By using Theorem \ref{Hooleyandthelargesieve1} as well as \eqref{IntroF--I}, we obtain both the slower convergence in the sum over $n_0$ seen in \eqref{mediumconductortrivialbound} and the heuristic convergence of the sums over $m$ seen in \eqref{standardF--Ibound}.

\subsection{Character sums over hyperbolic regions} The secondary purpose of this paper is to provide asymptotics for character sums with a certain hyperbolic height condition. Throughout, let $\|a_1,\ldots,a_n\|=\max\{\lvert a_1\rvert,\ldots,\lvert a_n\rvert\}$ for any $n\in\N$. For the duration of this exposition we set the notation $h(\b{k},\b{l})=\|n_0m_0,n_1m_1\|\cdot\|n_2m_2,n_3m_3\|$ for vectors  $\b{k},\b{l}\in\N^4$. The character sums we wish to consider are of the form
\begin{equation}\label{Introcharactersum}
    \mathop{\sum\sum}_{\substack{h(\b{k},\b{l})\leq B\\ \|n_0,m_0,n_1,m_1\|>1\\ \|n_2,m_2,n_3,m_3\|>1
    }}\left(\prod_{i=0}^{3}\frac{1}{\tau(n_i)\tau(l_i)}\right)\left(\frac{m_2m_3}{n_0n_1}\right)\left(\frac{m_0m_1}{n_2n_3}\right).
\end{equation}
Our purpose in studying such character sums arises from the following counting problem:
\begin{equation}\label{thecountingproblem}
N(B) = \#\left\{(y_0,y_1,y_2,y_3)\in\Z^4:\begin{array}{c} y_0y_1 = y_2y_3,\;Q_{\b{y}}(\Q)\neq \emptyset \\\mathrm{gcd}(y_0,y_1)=\mathrm{gcd}(y_2,y_3)=1,\\ -y_0y_2\neq\square\;\text{and}\; -y_0y_3\neq \square,\\\|y_0,y_1,y_2,y_3\|\leq B \end{array}\right\}
\end{equation}
where $Q_{\b{y}}$ denotes the rational diagonal quadric surface defined by the equation
\[
y_0x_0^2+y_1x_1^2+y_2x_2^2+y_3x_3^2 = 0
\]
and $Q_{\b{y}}(\Q)$ denotes the set of rational points of $\Q$. This counting problem is an example of a local solubility problem. Starting from work of Serre \cite{SerreConics}, there has been a growing body of research into providing bounds and asymptotics for such counting problems \cite{Bhargava,BH-B,BL,BLS,BLSmeets,DLS,F--I,Guo,Hooley1,Hooley2,KPSS,L,LRS,LS,LTbT,Sofos,SofosVisse}. Conjectures into the asymptotic order of growth have been given by Loughran and Smeets \cite{LS} and Loughran, Rome, and Sofos \cite{LRS}. The counting problem \eqref{thecountingproblem} is of particular interest as the family of quadrics $Q_{\b{y}}$ is parameterised by rational points on the quadric surface cut out by the equation $y_0y_1=y_2y_3$. If we allow points with $-y_0y_2=\square$ or $-y_0y_3=\square$, then it was shown in \cite{BLS} that
\begin{equation}\label{BLSIntro}
B^2\ll \#\left\{(y_0,y_1,y_2,y_3)\in\Z^4:\begin{array}{c} y_0y_1 = y_2y_3,\;Q_{\b{y}}(\Q)\neq \emptyset \\\mathrm{gcd}(y_0,y_1)=\mathrm{gcd}(y_2,y_3)=1,\\\|y_0,y_1,y_2,y_3\|\leq B \end{array}\right\} \ll B^2.
\end{equation}
Although this particular counting problem does not fall under the aforementioned conjectures, the asymptotic order of $B^2$ is larger than what would be expected by following the same philosophy. Further, they associate the abundance of soluble fibres to the ``thin'' set of points $(y_0,y_1,y_2,y_3)$ such that $-y_0y_2=\square$ or $-y_0y_3=\square$. One therefore wonders if the asymptotic order $\frac{B^2}{\log B}$ is true once we remove this set of points. In an adjacent paper, \cite{MeQuadricsPaper}, we prove this to be false by proving the true asymptotic:
\begin{theorem}[\cite{MeQuadricsPaper}, Theorem $1.1$]\label{quadricstheorem}
    For sufficiently large $B$,
    \[
    N(B)\sim \frac{cB^2\log\log B}{\log B}
    \]
    for some positive constant $c>0$.
\end{theorem}
After various reductions, the proof of this theorem will hinge on proving asymptotics for character sums of the form \eqref{Introcharactersum} (note that the condition $-y_0y_2\neq\square$ and $-y_0y_3\neq\square$ in \eqref{thecountingproblem} translates to the $\|n_0,m_0,n_1,m_1\|>1$ and $\|n_0,m_0,n_1,m_1\|>1$ conditions \eqref{Introcharactersum}). By decomposing $\N^{8}$ appropriately \eqref{Introcharactersum} sum will split into smaller sums, each of which will be one of the following three types:
\begin{align}
&\mathop{\sum\sum}_{\substack{h(\b{k},\b{l})\leq B \\ \|m_0,m_1,m_2,m_3\|\leq (\log B)^{1000}\\\|n_0,m_0,n_1,m_1\|>1\\ \|n_2,m_2,n_3,m_3\|>1}}\left(\prod_{i=0}^{3}\frac{1}{\tau(n_i)\tau(l_i)}\right)\left(\frac{m_2m_3}{n_0n_1}\right)\left(\frac{m_0m_1}{n_2n_3}\right),\label{Introanalysissum1}\\ &\mathop{\sum\sum}_{\substack{h(\b{k},\b{l})\leq B \\ \|n_2,m_2,n_3,m_3\|\leq (\log B)^{1000}\\\|n_0,m_0,n_1,m_1\|>1\\ \|n_2,m_2,n_3,m_3\|>1}}\left(\prod_{i=0}^{3}\frac{1}{\tau(n_i)\tau(l_i)}\right)\left(\frac{m_2m_3}{n_0n_1}\right)\left(\frac{m_0m_1}{n_2n_3}\right),\label{Introanalysissum2}\\
&\mathop{\sum\sum}_{\substack{h(\b{k},\b{l})\leq B \\ \|n_0,n_1\|,\|m_2,m_3\|>(\log B)^{1000}}}\left(\prod_{i=0}^{3}\frac{1}{\tau(n_i)\tau(l_i)}\right)\left(\frac{m_2m_3}{n_0n_1}\right)\left(\frac{m_0m_1}{n_2n_3}\right).\label{Introanalysissum3}
\end{align}
In \S\ref{hyperbolicanalysis1} and \S\ref{hyperbolicanalysis2} we deal sums of the form \eqref{Introanalysissum1} and \eqref{Introanalysissum2}. The main tool for sums of the form \eqref{Introanalysissum3} is a previous result of the author, \cite[Theorem $1.1$]{Me}, and so we will not consider sums of this form in the current paper. How the results of \S\ref{hyperbolicanalysis1}, \S\ref{hyperbolicanalysis2} and \cite[Theorem $1.1$]{Me} combine to give Theorem \ref{quadricstheorem} is the topic of the companion paper \cite{MeQuadricsPaper}.\\

The basic strategy for handling sums of the form \eqref{Introanalysissum1} and \eqref{Introanalysissum2} is to use the hyperbola method on the height condition $h(\b{k},\b{l})$ and apply Selberg--Delange methods (see Lemma \ref{Siegel--Walfisz1}). This strategy will result in issues, as we will come across hyper-skewed regions where current Selberg--Delange results become worse than the trivial bound. Furthermore, the usual partners of Selberg--Delange results in character sum methods, the large sieve \eqref{IntroHBbound} and \eqref{IntroF--I} also fail to give satisfactory bounds for these error terms. We thus employ Theorem \ref{Hooleyandthelargesieve1} in the guise of Corollaries \ref{LfunctionweirdaverageIntro} and \ref{AveragingovermediumconductorsIntro} to handle these error terms.\\

In \S\ref{technicalstuff} we gather all of the results required to prove the results of \S\ref{hyperbolicanalysis1} and \S\ref{hyperbolicanalysis2}. In \S\ref{hyperbolicanalysis1} we consider sums of the form \eqref{Introanalysissum1}, and in the final section, \S\ref{hyperbolicanalysis2}, we consider sums of the form \eqref{Introanalysissum2}. The statements of the results of these latter sections are rather complicated, as they are designed to be of necessary generality for their application to the proof of Theorem \ref{quadricstheorem}. The main results of these sections are Propositions \ref{maintermproposition}, \ref{symmetrictypeaverage1} and \ref{Asymmetrictypeaverage1} from \S\ref{hyperbolicanalysis1} and Lemma \ref{fixedconductorlemmaforvanishingmainterms} and Propositions \ref{symmetrictypeaverage2} and \ref{Asymmetrictypeaverage2} from \S\ref{hyperbolicanalysis2}.

Lastly, we remark that the methods of this paper may also be used to consider \eqref{Introcharactersum} without the extra conditions that $\|n_0,m_0,n_1,m_1\|>1$, $\|n_2,m_2,n_3,m_3\|>1$ (i.e character sums corresponding to finding the asymptotic formula for \ref{BLSIntro} as opposed to upper and lower bounds). By doing so, we will obtain variants of \eqref{Introanalysissum2} in which we allow $n_0=m_0=n_1=m_1=1$, which would be of order $B^2$. As a result, computing asymptotics for \eqref{Introcharactersum} without $\|n_0,m_0,n_1,m_1\|>1$ and $\|n_2,m_2,n_3,m_3\|>1$ would not require error terms which are are as strong as we provide through the use of Theorem \ref{Hooleyandthelargesieve1} (for example, \eqref{mediumconductortrivialbound} would suffice instead of the Corollary \ref{AveragingovermediumconductorsIntro}). With the condition that $\|n_0,m_0,n_1,m_1\|>1$, $\|n_2,m_2,n_3,m_3\|>1$, the main term of \eqref{Introcharactersum} no longer comes from regions of the form \eqref{Introanalysissum2}, but from \eqref{Introanalysissum1} - in particular the sections where $n_0=n_1=n_2=n_3=1$ or $m_0=m_1=m_2=m_3=1$.

\subsection{Acknowledgements} I would like to thank my Ph.D. supervisor, Efthymios Sofos, for his guidance throughout this project, and to Tim Browning for many helpful comments on the introduction. Lastly, I am grateful to The Carnegie Trust for the Universities of Scotland for sponsoring my Ph.D.

\section{Hooley Neutralisers and the Large Sieve}\label{HooleyLargeSieve}
In this section will prove Theorem \ref{Hooleyandthelargesieve1}. This is done with the use of Hooley neutralisers. We will begin with the following lemma, which is a slight modification of \cite[Proposition $4.1$]{Sofos} and has a similar proof.

\begin{lemma}\label{neutralisers}
    Let $P(z)$ be the product of all odd primes $\leq z$ and suppose $(\lambda^{+}_d)$ is any sequence satisfying
    \begin{equation}\label{upperboundsequencecondition}
    \sum_{d|n}\lambda^{+}_d \geq \mathds{1}(n=1).
    \end{equation}
    For any function $f:\N\rightarrow [0,1]$ such that
    \begin{itemize}
        \item[(1)] $f$ is multiplicative,
        \item[(2)] $f(p^m)\leq f(p)$ for all primes $p$ and all $m\geq 1$,
    \end{itemize}
    define the multiplicative function $\hat{f}:\N\rightarrow\R$ by $\hat{f}(n) = \prod_{p|n}(1-f(p))$. Then for all integers $n$:
    \[
    f(n) \leq \sum_{\substack{d|n\\d|P(z)}}\lambda_d^{+}\hat{f}(d).
    \]
\end{lemma}

\begin{proof}
     We let $n$ be a square-free integer composed only of primes $p\leq z$, i.e. $n|P(z)$. Then
    \[
    f(n) = \sum_{m|n}\hat{f}(m)f\left(\frac{n}{m}\right)\mathds{1}(m=1).
    \]
    Since $f$ and $\hat{f}$ are non-negative, we may use \eqref{upperboundsequencecondition} to get the upper bound:
    \[
    f(n)\leq \sum_{m|n}\hat{f}(m)f\left(\frac{n}{m}\right)\left(\sum_{d|m}\lambda_d^{+}\right).
    \]
    By writing $m=dd'$ and reversing the order of summation we get
    \[
    \sum_{m|n}\hat{f}(m)f\left(\frac{n}{m}\right)\left(\sum_{d|m}\lambda_d^{+}\right) = \sum_{d|n}\lambda_{d}^{+}\left(\sum_{d'|\frac{n}{d}}\hat{f}(dd')f\left(\frac{n}{dd'}\right)\right) = \sum_{d|n}\lambda_{d}^{+}\hat{f}(d)\left(\sum_{d'|\frac{n}{d}}\hat{f}(d')f\left(\frac{n}{dd'}\right)\right),
    \]
    this last equality being obtained by noting since $n$ is square-free and $dd'|n$, $d$ and $d'$ must be square-free and co-prime. Now
    \[
    \sum_{d'|\frac{n}{d}}\hat{f}(d')f\left(\frac{n}{dd'}\right) = \hat{f}*f\left(\frac{n}{d}\right)
    \]
    where $*$ denotes Dirichlet convolution. Then since $f$ and $\hat{f}$ are multiplicative, $\hat{f}*f$ will also be multiplicative. However, $\hat{f}*f(p)=1$ for all primes $p$ since $\hat{f}(p) = (1-f(p))$. Therefore $\hat{f}*f(n)=1$ for all square-free integers $n$. Thus
    \[
    \sum_{d'|\frac{n}{d}}\hat{f}(d')f\left(\frac{n}{dd'}\right) = 1.
    \]
    It follows that
    \[
    f(n)\leq \sum_{d|n}\lambda_{d}^{+}\hat{f}(d).
    \]
    Since we assumed that $n|P(z)$ we are done in this case. For more general integers $n$ we write:
    \[
    n = \left(\prod_{\substack{p|n\\p|P(z)}} p^{v_p(n)}\right) \left(\prod_{\substack{q|n\\q\nmid P(z)}}q^{v_q(n)}\right).
    \]
    Then, since $f$ is multiplicative and has image in $[0,1]$ we use assumption $(2)$ to obtain:
    \[
    f(n) \leq f\left(\prod_{\substack{p|n\\p|P(z)}} p\right) \leq \sum_{\substack{d|n\\ d|P(z)}}\lambda_{d}^{+}\hat{f}(d),
    \]
    where we use the previous case in this last inequality.
\end{proof}

We will be interested in using this for upper bound sieve coefficents. Fix some $z>0$ and let $y=z^{10}$. Then in particular we define the upper bound sieve coefficients
\begin{equation}\label{sievecoefficients}
\lambda_{d}^{+} = \mathds{1}(d\in\mathcal{D}^{+})\mu(d)
\end{equation}
where
\[
\mathcal{D}^{+} = \{d=p_1\ldots p_k\in\N:z>p_1>\ldots>p_k,\;p_m<y_m\;\text{for $m$ odd} \}
\]
for $y_m = \left(\frac{y}{p_1\ldots p_m}\right)^{1/\beta}$ and some $\beta>1$. It is well known that
\begin{equation}\label{delta}
    \sum_{d|n}\lambda_{d}^{+} \geq \sum_{d|n}\mu(d) = \mathds{1}(n=1)
\end{equation}
and that the $\lambda_d^+$ are supported on the interval $[1,y]$. We note also that any multiplicative function $f$ satisfying the conditions of Lemma \ref{neutralisers} will also satisfy conditions $(i)$ and $(ii)$ of \cite[Section $2$]{Shiu}. We may then use \cite[Theorem $1$]{Shiu} with, $k=1$, $Y=X$ to obtain the bound
\begin{equation}\label{Brun-Titchmarshformultiplicative}
    \sum_{n\leq X}f(n) \ll \frac{X}{\log X}\exp\left(\sum_{p\leq X}\frac{f(p)}{p}\right).
\end{equation}
Furthermore, we will require that our multiplicative functions satisfy,
\[
\sum_{p\leq X}\frac{f(p)}{p} = \alpha\log\log X + O(1)
\]
for some $\alpha>0$. With this in mind, \eqref{Brun-Titchmarshformultiplicative} becomes,
\begin{equation}\label{Brun-Titchmarshformultiplicative2}
    \sum_{n\leq X}f(n) \ll \frac{X}{(\log X)^{1-\alpha}}.
\end{equation}
The following lemma encodes the insertion of the Brun Sieve into the large sieve,
\begin{lemma}\label{sieveapplication}
    Let $X\geq 2$. Fix some $\epsilon>0$ and set $z=X^{\epsilon/10}$ and $y=X^{\epsilon}$. Let $f$ be any function $f:\N\rightarrow [0,1]$ such that
    \begin{itemize}
        \item[(1)] $f$ is multiplicative,
        \item[(2)] $f(p^m)\leq f(p)$ for all primes $p$ and all $m\geq 1$.
    \end{itemize}
    Suppose also that there exists some $0<\alpha\leq 1$ such that, for all $2\leq Y$,
    \begin{equation}\label{kindamerten}
    \sum_{p\leq Y}\frac{f(p)}{p} = \alpha\log\log Y + O(1).
    \end{equation}
    Then for the sieve coefficients $(\lambda_d^+)$ defined above and any integer $n$,
    \[
    \left\lvert\sum_{\substack{d\leq X\\ d|P(z)\\ \mathrm{gcd}(d,n)=1}}\lambda_d^{+}f(d) \sum_{\substack{m\leq X\\d|m\\ \mathrm{gcd}(m,n)=1}} 1\right\rvert \ll_{\epsilon} \frac{\phi(n)X}{n(\log X)^{\alpha}}\prod_{\substack{p|P(z)\\ p|n}}\left(1-\frac{f(p)}{p}\right)^{-1} + X^{\epsilon}\tau(n)
    \]
    where the implied constant depends at most on $\epsilon$.
\end{lemma}

\begin{proof}
Using the fact that $\lambda_d^{+}$ is supported on $[1,y]$ we have
\begin{align*}
    \sum_{\substack{d\leq X\\ d|P(z)\\ \mathrm{gcd}(d,n)=1}}\lambda_d^{+}f(d) \sum_{\substack{m\leq X\\d|m\\ \mathrm{gcd}(m,n)=1}} 1 &= \frac{X\phi(n)}{n}\sum_{\substack{d\leq y\\ d|P(z)\\ \mathrm{gcd}(d,n)=1}} \lambda_d^{+}\frac{f(d)}{d} + O(y\tau(n))\\
    &= \frac{X\phi(n)}{n}\sum_{\substack{d|P(z)\\ \mathrm{gcd}(d,n)=1}} \lambda_d^{+}\frac{f(d)}{d} + O(y\tau(n)).
\end{align*}
Next, we would like to apply the fundamental lemma of sieve theory to the sum over $d$, \cite[Fundamental Lemma 6.3, pg.159]{Iwaniec--Kowalski}. In order to do so, we first need to satisfy the condition
\begin{align*}
\prod_{\substack{w\leq p<z\\ p\;\text{prime}}}\left(1-\frac{\mathds{1}(\mathrm{gcd}(n,d)=1)f(p)}{p}\right)^{-1} \leq K\left(\frac{\log z}{\log w}\right)
\end{align*}
for any $0<w\leq z$ where $K$ is some absolute constant. For this we note that, by assumption \eqref{kindamerten} on $f$ it follows that
\begin{align*}
\prod_{\substack{w<p\leq z\\ p\nmid n}}\left(1-\frac{f(p)}{p}\right)^{-1} &= \prod_{\substack{w<p\leq z\\p|n}}\left(1-\frac{f(p)}{p}\right)\prod_{w<p\leq z}\left(1-\frac{f(p)}{p}\right)^{-1},\\
&\ll \left(\frac{\log z}{\log w}\right)^{\alpha} \prod_{\substack{w<p\leq z\\p|n}}\left(1-\frac{f(p)}{p}\right) \ll \left(\frac{\log z}{\log w}\right)^{\alpha}
\end{align*}
since $\prod_{\substack{p|n}}\left(1-\frac{f(p)}{p}\right)<1$ for all $n$. Thus we may apply the fundamental lemma to the sum over $d$ to obtain the upper bound:
\begin{align*}
    \sum_{\substack{d|P(z)}} \lambda_d^{+}\frac{\mathds{1}(\mathrm{gcd}(n,d)=1)f(d)}{d} \ll& \prod_{p|P(z)}\left(1-\frac{\mathds{1}(\mathrm{gcd}(n,d)=1)f(p)}{p}\right) \\
    \ll& \prod_{\substack{p|P(z)\\p|n}}\left(1-\frac{f(p)}{p}\right)^{-1}\prod_{\substack{p|P(z)}}\left(1-\frac{f(p)}{p}\right)\\
    \ll& \frac{1}{(\log z)^{\alpha}}\prod_{\substack{p|P(z)\\p|n}}\left(1-\frac{f(p)}{p}\right)^{-1}.
\end{align*}
Recalling that $z=X^{\epsilon/10}$ and $y=X^{\epsilon}$, we substitute this into our equalities above to obtain the result.
\end{proof}
We now prove the main result of this section.

\begin{proof}[Proof of Theorem \ref{Hooleyandthelargesieve1}]
Set the sum on the left to be $S(N,M)$. Then by the Cauchy--Schwarz inequality:
\[
S(N,M)^2 \ll \left(\sum_{m\leq M} \lvert b_m\rvert f(m)\right) \left(\sum_{m\leq M}\lvert b_m\rvert f(m) \left\lvert \sum_{n\leq N}a_n\left(\frac{n}{m}\right)\right\rvert^2\right)
\]
The first sum over $m$ is bounded using \eqref{Brun-Titchmarshformultiplicative2}:
\[
\sum_{m\leq M}\lvert b_m\rvert f(m) \ll \sum_{m\leq M}f(m) \ll \frac{M}{(\log M)^{1-\alpha}}.
\]
Fix $z=M^{\epsilon/10}$ and $y=z^{10}=M^{\epsilon}$. Using Lemma \ref{neutralisers} we have
\[
\lvert b_m \rvert f(m) \leq f(m) \leq \sum_{\substack{d|m\\d|P(z)}}\lambda_d^{+}\hat{f}(d)
\]
where $\hat{f}$ is as defined in Lemma \ref{neutralisers} and $(\lambda_d^+)$ are the sieve coefficients defined in \eqref{sievecoefficients}. Therefore,
\begin{align*}
\sum_{m\leq M}\lvert b_m\rvert f(m)\left\lvert \sum_{n\leq N}a_n\left(\frac{n}{m}\right)\right\rvert^2 \leq& \sum_{m\leq M}\sum_{\substack{d|m\\ d|P(z)}}\lambda_d^{+}\hat{f}(d)\left\lvert \sum_{n\leq N}a_n\left(\frac{n}{m}\right)\right\rvert^2\\ \leq& \sum_{\substack{d\leq M\\ d|P(z)}}\lambda_d^{+}\hat{f}(d)\sum_{\substack{m\leq M\\ d|m}}\left\lvert \sum_{n\leq N}a_n\left(\frac{n}{m}\right)\right\rvert^2.
\end{align*}
Now, by expanding the square this becomes:
\begin{align*}
    \sum_{m\leq M}\lvert b_m\rvert f(m) \left\lvert \sum_{n\leq N}a_n\left(\frac{n}{m}\right)\right\rvert^2 \leq& \sum_{\substack{d\leq M\\ d|P(z)}}\lambda_d^{+}\hat{f}(d)\sum_{\substack{m\leq M\\ d|m}}\left(\sum_{n_1,n_2\leq N}a_{n_1}\bar{a}_{n_2}\left(\frac{n_1n_2}{m}\right)\right)\\
    \leq& \sum_{\substack{n_1,n_2\leq N\\ n_1n_2=\square}}a_{n_1}\bar{a}_{n_2} \sum_{\substack{d\leq M\\ d|P(z)}}\lambda_d^{+}\hat{f}(d) \sum_{\substack{m\leq M\\d|m\\ \mathrm{gcd}(m,n_1n_2)=1}} 1\\ &+ \sum_{\substack{n_1,n_2\leq N\\ n_1n_2\neq\square}}a_{n_1}\bar{a}_{n_2} \sum_{\substack{d\leq M\\ d|P(z)}}\lambda_d^{+}\hat{f}(d) \sum_{\substack{m\leq M\\ d|m}} \left(\frac{n_1n_2}{m}\right).
\end{align*}
We first consider the sum over $n_1n_2\neq \square$. Here we write $m=m'd$ and note that the sieve coefficents $\lambda_d^{+}$ are supported on the interval $[1,y]$. Then, using the P\'olya--Vinogradov inequality:
\begin{align*}
    \sum_{\substack{n_1,n_2\leq N\\ n_1n_2\neq\square}}a_{n_1}\bar{a}_{n_2}\sum_{\substack{d\leq M\\ d|P(z)}}\lambda_d^{+}\hat{f}(d) \sum_{\substack{m\leq M\\ d|m}} \left(\frac{n_1n_2}{m}\right) &\ll \sum_{\substack{n_1,n_2\leq N\\ n_1n_2\neq\square}} \sum_{d\leq y}\left\lvert \sum_{\substack{m'\leq M/d}} \left(\frac{n_1n_2}{m'}\right)\right\rvert\\
    &\ll y\sum_{\substack{n_1,n_2\leq N\\ n_1n_2\neq\square}}n_1^{1/2}n_2^{1/2}(\log n_1n_2)\\
    &\ll yN^3(\log N).
\end{align*}
For the sum over the squares we have
\begin{align*}
    \sum_{\substack{n_1,n_2\leq N\\ n_1n_2=\square}}a_{n_1}\bar{a}_{n_2}\sum_{\substack{d\leq M\\ p|P(z)\\ \mathrm{gcd}(d,n_1n_2)=1}}\lambda_d^{+}\hat{f}(d) \sum_{\substack{m\leq M\\d|m\\ \mathrm{gcd}(m,n_1n_2)=1}} 1 \ll& \sum_{\substack{n\leq N}}\tau(n^2)\left\lvert\sum_{\substack{d\leq M\\ d|P(z)\\ \mathrm{gcd}(d,n)=1}}\lambda_d^{+}\hat{f}(d) \sum_{\substack{m\leq M\\d|m\\ \mathrm{gcd}(m,n)=1}} 1\right\rvert.
\end{align*}
Note that, by assumption \eqref{Mertenassumption} on $f$, and the definition of $\hat{f}$, we may write,
\[
\sum_{p\leq X}\frac{\hat{f}(p)}{p} = \sum_{\substack{p\leq X}}\frac{1-f(p)}{p} = (1-\alpha)\log\log X + O(1)
\]
for $2\leq X$. Note also that $\hat{f}(p^m)=\hat{f}(p)$ for all primes $p$ and all $m\geq 1$. It follows that $\hat{f}$ satisfies the conditions of Lemma \ref{sieveapplication}. Thus, we may use this lemma to bound the inner sum by
\[
    \left\lvert\sum_{\substack{d\leq M\\ d|P(z)\\ \mathrm{gcd}(d,n)=1}}\lambda_d^{+}\hat{f}(d) \sum_{\substack{m\leq M\\d|m\\ \mathrm{gcd}(m,n)=1}} 1\right\rvert \ll_{\epsilon} \frac{\phi(n)Mg_z(n)}{n(\log M)^{1-\alpha}} + O(y\tau(n)),
\]
where we have set
\[
g_z(n) = \prod_{\substack{p|P(z)\\p|n}}\left(1-\frac{\hat{f}(p)}{p}\right)^{-1}.
\]
Substituting this into the sum over $n_1n_2=\square$ we get,
\begin{align*}
    \mathop{\sum\sum}_{\substack{n_1,n_2\leq N\\ n_1n_2=\square}}a_{n_1}\bar{a}_{n_2}\hspace{-5mm}\sum_{\substack{d\leq M\\ p|P(z)\\ \mathrm{gcd}(d,n_1n_2)=1}}\hspace{-5mm}\lambda_d^{+}\hat{f}(d) \hspace{-5mm}\sum_{\substack{m\leq M\\d|m\\ \mathrm{gcd}(m,n_1n_2)=1}} 1 \ll_{\epsilon}& \left(\frac{M}{(\log M)^{1-\alpha}}\sum_{n\leq N}\tau(n^2)g_z(n)+y\sum_{n\leq N}\tau(n^2)\tau(n)\right)\\
    \ll_{\epsilon}& \frac{MN(\log N)^2}{(\log M)^{1-\alpha}} + yN(\log N)^5.
\end{align*}
where we have used \eqref{Brun-Titchmarshformultiplicative} to obtain
\begin{align*}
     \sum_{n\leq N}\tau(n^2)g_z(n) &\ll N(\log N)^2\;\text{and}\;\sum_{n\leq N}\tau(n^2)\tau(n)\ll N(\log N)^5
\end{align*}
since $g_z(p)\leq 2$ for all primes $p$. To conclude, we now have
\begin{align*}
S(N,M)^{2} &\ll_{\epsilon} \frac{M}{(\log M)^{1-\alpha}}\left(\frac{MN(\log N)^2}{(\log M)^{1-\alpha}} + yN(\log N)^5 + yN^{3}(\log N)\right)\\
&\ll_{\epsilon} \frac{M^2N(\log N)^2}{(\log M)^{1-\alpha}(\log M)^{1-\alpha}} + \frac{yMN^{3}(\log N)}{(\log M)^{1-\alpha}}
\end{align*} 
Taking square roots and noting that $(x+y)^{1/2}\ll (x^{1/2}+y^{1/2})$ for $x,y\geq 0$ gives the result, since $y=M^{\epsilon}$.
\end{proof}

\begin{proof}[Proof of Corollary \ref{Hooleyandthelargesieve2}]
Set
\[
S_f(u,M) = \mathop{\sum\sum}_{\substack{n\leq u, m\leq M}}a_nb_mf(m)\left(\frac{n}{m}\right).
\]
Then, by partial summation in the variable $n$, we obtain
\[
\sum_{W<n\leq N}\sum_{m\leq M}\frac{a_n}{n}b_mf(m)\left(\frac{n}{m}\right) = \frac{S_f(N,M)}{N}-\frac{S_f(W,M)}{W} + \int_{W}^{N}\frac{S_f(u,M)}{u^2}\mathrm{d}u.
\]
The result follows from applying Theorem \ref{Hooleyandthelargesieve1}, computing the integral, and comparing the sizes of the corresponding terms.
\end{proof}

\section{Technical Lemmas}\label{technicalstuff}
In this section, we record and prove some technical lemmas that will be heavily used moving forward.

\subsection{Large Conductor Lemmas}
We will need the following modification to \eqref{IntroF--I}:
\begin{lemma}\label{Harmonic-Friedlander--Iwaniec}
Let $N,M\geq 2$ and $W<N,M$. Suppose $(a_n)$, $(b_m)$ are any complex sequences supported on the odd integers such that $\lvert a_n\rvert,\lvert b_m\rvert\leq 1$.
\[
\sum_{W<n\leq N}\sum_{m\leq M} \frac{a_n}{n} b_m\left(\frac{n}{m}\right) \ll \frac{M(\log 3NM)^{7/6}}{W^{1/6}}. 
\]
\end{lemma}

\begin{proof}
    For this proof, set 
    \[
    S(u,M) = \sum_{n\leq u}\sum_{m\leq M}a_n b_m \left(\frac{n}{m}\right).
    \]
    Using partial summation in the $n$ variable we have
    \[
    \sum_{W<n\leq N}\sum_{m\leq M} \frac{a_n}{n} b_m\left(\frac{n}{m}\right) = \frac{S(N,M)}{N} - \frac{S(W,M)}{W} + \int_{W}^{N}\frac{S(u,M)}{u^2}\mathrm{d}u.
    \]
    Then bounding $S(u,M)$ using \eqref{IntroF--I}, and noting that $M,N\geq W$, we may obtain the result by trivially bounding the logarithms and computing the integral.
\end{proof}

This result is particularly useful when $N\leq \frac{M}{W}$. Next, we will need the following special case of the results in \S \ref{HooleyLargeSieve}:

\begin{lemma}\label{Hooleyandthelargesieve3}
    Let $M\geq N\geq W\geq 2$, and fix some $\epsilon>0$. For any complex sequences $a_n$,$b_m$ which are supported on the odd integers such that $\lvert a_n\rvert\leq 1$ and $\lvert b_m\rvert\leq 1$ we have:
    \[
    \sum_{W<n\leq N}\sum_{m\leq M} \frac{a_n b_m}{n\tau(m)} \left(\frac{n}{m}\right) \ll_{\epsilon} \frac{M(\log N)}{W^{1/2}(\log M)^{1/2}} + \frac{M^{1/2+\epsilon}N^{1/2}(\log N)^{1/2}}{(\log M)^{1/4}},
    \]
    where the implied constant depends at most on $\epsilon$.
\end{lemma}

\begin{proof}
This result is a straightforward application of Corollary \ref{Hooleyandthelargesieve2}. Indeed $\frac{1}{\tau}$ is a multiplicative function satisfying all conditions of this corollary. In particular, it satisfies assumption \eqref{Mertenassumption} with $\alpha=1/2$, giving the result.
\end{proof}

This result is particularly useful in regions where $N^2\leq M$.

\subsection{Small Conductor Lemmas}
In order to deal with regions where our sums involve Jacobi symbols which have small conductors, we will require Siegel--Walfisz methods. Define
\begin{equation}\label{TheConstants}
     f_0 = \frac{1}{\sqrt{\pi}}\prod_{p\;\;\text{prime}}f_p\left(1-\frac{1}{p}\right)^{1/2}\;\;\;\text{and}\;\;\;f_p = 1+\sum_{j=1}^{\infty}\frac{1}{(j+1)p^{j}}.
\end{equation}
Our key lemma for this purpose is Lemma $5.9$ from \cite{LRS}:

\begin{lemma}\label{Siegel--Walfisz1}
    Let $r$ and $Q$ be integers such that $\mathrm{gcd}(r,Q)=1$. Let $\chi(n)$ be a character modulo $Q$. Fix any $C>0$. Then for all $X\geq 2$ we have:
    \[
    \sum_{\substack{n\leq X\\ \mathrm{gcd}(n,r)=1}}\frac{\chi(n)}{\tau(n)} = \frac{\mathds{1}(\chi=\chi_0)\mathfrak{S}_0(Qr)X}{\sqrt{\log X}}\left\{1+O\left(\frac{(\log\log3rQ)^{3/2}}{\log X}\right)\right\}+O_C\left(\frac{\tau(r)QX}{(\log X)^C}\right)
    \]
    where $\chi_0$ is the principal character modulo $Q$ and
    \[
    \mathfrak{S}_0(Qr) = \frac{f_0}{\left(\prod_{p\mid 2rQ}f_p\right)}.
    \]
    Furthermore, if $Q$ and $r$ are odd and $q\in(\Z/8\Z)^{*}$, then we have
    \[
    \sum_{\substack{n\leq X\\ \mathrm{gcd}(n,r)=1\\ n\equiv q\bmod{8}}}\frac{\chi(n)}{\tau(n)} = \frac{\mathds{1}(\chi=\chi_0)\mathfrak{S}_0(Qr)X}{\phi(8)\sqrt{\log X}}\left\{1+O\left(\frac{(\log\log3rQ)^{3/2}}{\log X}\right)\right\}+O_C\left(\frac{\tau(r)QX}{(\log X)^C}\right)
    \]
\end{lemma}

\begin{proof}
We may use the same contour argument performed for Lemma $1$ in \cite{F--I}, along with the fact that, if $\chi\neq\chi_0$, the Dirichlet series
\[
\sum_{\substack{n\in\N \\ \mathrm{gcd}(n,r)=1}}\frac{\chi(n)}{\tau(n)n^s}
\]
is holomorphic for $\Re(s)>1-\frac{c(\epsilon)}{Q^{\epsilon}(\log \Im(s))}$ for any $\epsilon>0$, where $c(\epsilon)$ is some positive constant depending only on $\epsilon$. For full details see \cite[Lemma 1]{F--I}. For the second part we have
\[
\sum_{\substack{n\leq X\\ \mathrm{gcd}(n,r)=1\\ n\equiv q\bmod{8}}}\frac{\chi(n)}{\tau(n)} = \frac{1}{\phi(8)}\sum_{\substack{\chi' \text{char.}\\ \bmod{8}}}\overline{\chi'}(q)\sum_{\substack{n\leq X\\ \mathrm{gcd}(n,r)=1}}\frac{\chi(n)\chi'(n)}{\tau(n)},
\]
and so this follows via an application of the first part.
\end{proof}

This result with $\chi=\chi_0$ is used to obtain the main term in Section $4$ of \cite{LRS}. We shall do the same, however, due to difficulties arising from our height conditions, we will also require the use of this result to obtain part of our error term (see Sections \ref{hyperbolicanalysis1} and \ref{hyperbolicanalysis2}).

\subsection{The hyperbola method} Finally, we note the following variant of the hyperbola method, which will allow us to handle the hyperbolic height condition in our character sums: 

\begin{lemma}\label{hyperbolamethod}
    Let $X\geq 2$ and let $2\leq Y\leq X^{1/2}$. Fix some constants $c_0,c_1,c_2,c_3\in\N$. Then for any functions $g_0,g_1,g_2,g_3:\N\rightarrow\C$, we have
    \begin{align*}    \mathop{\sum\sum\sum\sum}_{\substack{n_0,n_1,n_2,n_3\in\N^4\\ \|n_0c_0,n_1c_1\|\cdot\|n_2c_2,n_3c_3\|\leq X}} \left(\prod_{i=0}^3 g_i(n_i)\right) =& \mathop{\sum\sum\sum\sum}_{\substack{\|n_0c_0,n_1c_1\|\leq Y\\ \|n_2c_2,n_3c_3\|\leq X/\|n_0c_0,n_1c_1\|}} \left(\prod_{i=0}^3 g_i(n_i)\right) \\+& \mathop{\sum\sum\sum\sum}_{\substack{\|n_2c_2,n_3c_3\|\leq X/Y\\ \|n_0c_0,n_1c_1\|\leq X/\|n_2c_2,n_3c_3\|}} \left(\prod_{i=0}^3 g_i(n_i)\right) \\ -&  \mathop{\sum\sum\sum\sum}_{\substack{\|n_0c_0,n_1c_1\|\leq Y\\ \|n_2c_2,n_3c_3\|\leq X/Y}} \left(\prod_{i=0}^3 g_i(n_i)\right).
    \end{align*}
\end{lemma}

\section{Applications of the neutraliser large sieve}
In this section, we provide the proofs of Corollaries \ref{LfunctionweirdaverageIntro} and \ref{AveragingovermediumconductorsIntro}.

\subsection{A general lemma}
For an odd integer $m$ we will denote by $\psi_{m}$ either the Jacobi symbol $\left(\frac{\cdot}{m}\right)$ or $\left(\frac{m}{\cdot}\right)$. Note that, for all $n$ odd, we may use a standard quadratic reciprocity to interchange between the two choices of $\psi_{m}(n)$. We will prove the following:
\begin{lemma}\label{Lfunctionweirdaverage}
    Let $\chi$ be a character modulo $q$, and $m_1$ be an odd integer in $[1,X]$. Suppose that $f$ is a multiplicative function satisfying the conditions of Theorem \ref{Hooleyandthelargesieve1}. Then for all $X\geq 3$ we have
    \[
    \mathop{\sum}_{\substack{1<m\leq X}}\mu^2(qmm_1)f(m)L(1,\chi\cdot\psi_{m m_1}) \ll_q \frac{X}{(\log X)^{1-\alpha}}
    \]
    where $\alpha$ is as defined by \eqref{Mertenassumption}and the implied constant depends at most on $q$.
\end{lemma}

\begin{proof}
    We begin by splitting the $L$-function in two:
    \[
    L(1,\chi\cdot\psi_{mm_1}) = \sum_{n=1}^{ X^2}\frac{\chi(n)\psi_{mm_1}(n)}{n} + \sum_{n> X^2}\frac{\chi(n)\psi_{mm_1}(n)}{n}.
    \]
    Using partial summation and the P\'olya--Vinogradov inequality \cite{Iwaniec--Kowalski}, the tail sum may be seen to be
    \[
    \ll \frac{(qmm_1)^{1/2}\log(mm_1)}{X^2} \ll \frac{q^{1/2}(\log X)}{X}.
    \]
    Summing trivially over $m$ will give $O(\log X)$ which is sufficient. The expression we have left is
    \begin{equation}\label{SmallpartofLfunction}    \sum_{\substack{m\leq X}}\sum_{n\leq X^2} \mu^{2}(2mm_1)f(m)\frac{\chi(n)\psi_{m_1}(n)}{n}\left(\frac{n}{m}\right).
    \end{equation}
    We will make use of the double oscillation of the character $\left(\frac{n}{m}\right)$ in the variables $m$ and $n$. Note that we have made one choice of definition for $\psi_{mm_1}$ - by using quadratic reciprocity, we may interchange between the two at this stage. By partial summation we get
    \begin{equation}\label{doubleoscillationafterpartialsummation}
    \sum_{\substack{m\leq X}}\sum_{n\leq \lfloor X^2\rfloor} \mu^{2}(2mm_1)f(m)\frac{\chi(n)\psi_{m_1}(n)}{n}\left(\frac{n}{m}\right)\ll \frac{S(X,X^2)}{X^2} + \int_{2}^{X^2}\frac{S(X,t)}{t^2}dt,
    \end{equation}
    where
    \[
    S(X,t) = \sum_{\substack{m\leq X}}\sum_{n\leq t} \mu^{2}(2mm_1)f(m)\chi(n)\psi_{m_1}(n)\left(\frac{n}{m}\right).
    \] 
    Using \eqref{IntroF--I} we obtain
    \[
    \frac{S(X,X^2)}{X^2} \ll \frac{X^3}{X^2}\left(X^{-1/6}+X^{-1/3}\right)(\log 3X)^{7/6}\ll X^{5/6}(\log 3X)^{7/6}.
    \]
    The integral equals:
    \[
    \int_{2}^{X^2}\frac{S(X,t)}{t^2}dt = \int_{2}^{X^{1/2}}\frac{S(X,t)}{t^2}dt + \int_{X^{1/2}}^{X^2}\frac{S(X,t)}{t^2}dt.
    \]
    In the first range we apply Corollary \ref{Hooleyandthelargesieve2} with $\epsilon=1/6$, $N=X$ and $M=t$ as $t^2\ll X$. In the second range we once more apply \eqref{IntroF--I}. The integral therefore becomes bounded by:
    \begin{align*}
    &\ll \int_{2}^{X^{1/2}}\left(\frac{X(\log t)}{t^{3/2}(\log X)^{1-\alpha}}+\frac{X^{1/3}(\log t)^{1/2}}{t^{1/2}(\log X)^{(1-\alpha)/2}}\right)dt + \int_{X^{1/2}}^{X^2}\left(\frac{X}{t^{7/6}}+\frac{X^{5/6}}{t}\right)(\log 3X)^{7/6}dt,
    \end{align*}
    which is $O(X(\log X)^{-(1-\alpha)})$.
\end{proof}

Corollary \ref{LfunctionweirdaverageIntro} now follows as a direct consequence of Lemma \ref{Lfunctionweirdaverage} with $m_1=1$, $f=\frac{1}{\tau}$ and $q=1$ (meaning $\chi(n)=1$ is for all $n\in\N$).

\subsection{Proof of Corollary \ref{AveragingovermediumconductorsIntro}}
For convenience we will write $Z=(\log X)^{C_1}$. Set
    \[
    T(X) = \sum_{\substack{(\log X)^{3C_1/4}<n_0c_0\leq X^{1/2}}}\frac{1}{n_0^2c_0^2\tau(n_0)}\left\lvert\mathop{\sum\sum}_{\substack{(C_1\log \log X)^{C_2}< m \leq (\log X)^{2C_1}\\ n_1c_1\leq n_0c_0}}\frac{a_m b_{n_1}}{m\tau(n_1)}\left(\frac{n_1}{m}\right)\right\rvert.
    \]
    Then we have
    \[
    T(X) \leq T_1(X) + T_2(X) + T_3(X)
    \]
    where
    \begin{align*}
        T_1(X) &= \sum_{\substack{Z^{10}<n_0\leq X^{1/2}/c_0}}\frac{1}{n_0^2c_0^2\tau(n_0)}\left\lvert\mathop{\sum\sum}_{\substack{(\log Z)^{C_2}< m \leq Z^{2}\\ n_1c_1\leq n_0c_0}}\frac{a_m b_{n_1}}{m\tau(n_1)}\left(\frac{n_1}{m}\right)\right\rvert,\\
        T_2(X) &= \sum_{\substack{Z^{3/4}/c_0<n_0\leq Z^{10
        }}}\frac{1}{n_0^2c_0^2\tau(n_0)}\left\lvert\mathop{\sum\sum}_{\substack{(\log Z)^{C_2}< m \leq Z^{1/10}\\n_1c_1\leq n_0c_0}}\frac{a_m b_{n_1}}{m\tau(n_1)}\left(\frac{n_1}{m}\right)\right\rvert,\\
        T_3(X) &= \sum_{\substack{Z^{3/4}/c_0<n_0\leq Z^{10}}}\frac{1}{n_0^2c_0^2\tau(n_0)}\left\lvert\mathop{\sum\sum}_{\substack{Z^{1/10}< m \leq Z^{2}\\n_1c_1\leq n_0c_0}}\frac{a_m b_{n_1}}{m\tau(n_1)}\left(\frac{n_1}{m}\right)\right\rvert.
    \end{align*}
    For $T_1(X)$ we apply Corollary \ref{Hooleyandthelargesieve3} with $\epsilon = \frac{1}{10}$. Then $T_1(X)$ is
    \begin{align*}
    \ll& \sum_{\substack{Z^{10}<n_0\leq X^{1/2}/c_0}}\frac{1}{n_0^2c_0^2\tau(n_0)} \left(\frac{n_0c_0(\log Z)}{c_1(\log Z)^{C_2/2}(\log n_0c_0/c_1)^{1/2}} + \frac{(n_0c_0)^{3/5}Z(\log Z)^{1/2}}{c_1^{3/5}(\log n_0c_0/c_1)^{1/4}}\right),\\
    \ll& \sum_{\substack{Z^{10}<n_0\leq X^{1/2}/c_0}}\left(\frac{(\log Z)}{n_0c_0c_1(\log Z)^{C_2/2}(\log n_0c_0/c_1)^{1/2}\tau(n_0)} + \frac{Z(\log Z)^{1/2}}{(n_0c_0)^{7/5}c_1^{3/5}(\log n_0c_0/c_1)^{1/4}\tau(n_0)}\right).
    \end{align*}
    Since $n_0c_0/c_1>Z^{10}c_0/c_1>2$, the second sum is
    \[
    \ll \frac{Z(\log Z)^{1/2}}{Z^{4}c_0^{7/5}c_1^{3/5}} \ll \frac{(\log Z)^{1/2}}{Z^{3}c_0^{7/5}c_1^{3/5}}.
    \]
    Using similar methods used to compute $M(X)$ in Lemma \ref{maintermlemma1}, we can bound the first sum by
    \begin{align*}
    \ll \sum_{\substack{n_0\leq X^{1/2}}}\frac{1}{n_0(\log n_0c_0/c_1)^{1/2}\tau(n_0)}\ll \log \log X.
    \end{align*}
    Substituting the previous two bounds into $T_1(X)$ will give
    \begin{align*}
    T_1(X) &\ll_{C_2} \frac{(\log \log X)}{c_0c_1(\log Z)^{C_2/2-1}} + \frac{(\log Z)^{1/2}}{Z^{3}c_0^{7/5}c_1^{3/5}} \ll_{C_2} \frac{(\log \log X)}{c_0c_1(\log Z)^{C_2/2-1}},
    \end{align*}
    using the fact that $Z>c_1^{2/5}$ to determine the dependence on $c_0$ and $c_1$. Using the same approach we can obtain the same bound for $T_2(X)$. Finally we deal with $T_3(X)$. Here it is better to apply Lemma \ref{Harmonic-Friedlander--Iwaniec} since the ranges of the inner double sum are of comparable size. Doing this we obtain,
    \begin{align*}
    T_3(X) &\ll \sum_{\substack{Z^{3/4}/c_0<n_0\leq Z^{10}}}\frac{(\log Z)^{7/6}}{n_0c_0c_1Z^{1/60}\tau(n_0)}
    \ll \frac{(\log Z)^{13/6}}{c_0c_1Z^{1/60}}, 
    \end{align*}
    which is sufficient.

\section{Character sums over hyperbolic regions I}\label{hyperbolicanalysis1}
In this section, we evaluate bounds for sums of a more general form to \eqref{Introanalysissum1}. Consider the general four variable sum below:
\begin{equation}\label{Generalsumtype}
\mathop{\sum\sum\sum\sum}_{\substack{\|n_0c_0,n_1c_1\|\cdot\|n_2c_2,n_3c_3\|\leq X\\ \mathrm{gcd}(n_i,2r_i)=1\;\forall\;0\leq i\leq 3\\ n_i\equiv q_i\bmod{8}\;\forall\;0\leq i\leq 3}} \frac{\chi_0(n_0)\chi_1(n_1)\chi_2(n_2)\chi_3(n_3)}{\tau(n_0)\tau(n_1)\tau(n_2)\tau(n_3)}
\end{equation}
where the sum is over $\b{n}\in\N^4$, $c_i,r_i,q_i\in\N$ are fixed, odd constants for each $0\leq i\leq 3$ and $\chi_i$ are some characters. Our methods vary depending on which of the characters are principal. In particular, we will have three cases to consider:
\begin{itemize}
    \item[(a)] Main Term: each $\chi_i$ is principal;
    \item[(b)] Small Conductor - Symmetric Hyperbola Method: $\chi_0$ or $\chi_1$ is non-principal and $\chi_2$ or $\chi_3$ is non-principal;
    \item[(c)] Small Conductor - Asymmetric Hyperbola Method: $\chi_0$ or $\chi_1$ is non-principal but $\chi_2$ and $\chi_3$ are principal or vice versa.
\end{itemize}
Each case will be handled using Lemmas \ref{Siegel--Walfisz1} and \ref{hyperbolamethod} - for case (c), we will also require the use of Corollary \ref{AveragingovermediumconductorsIntro}. In cases (b) and (c), we will also provide results that average over the conductors of the characters.

\subsection{Main Term}
We first provide asymptotics for  \eqref{Generalsumtype} when all of the $\chi_i$ are principal. First we will require some preliminary lemmas:

\begin{lemma}\label{maintermlemma1}
    Let $X\geq 3$, $C_1,C_2>0$ and take any $q_0,q_1\in\left(\Z/8\Z\right)^*$ . Then for any odd integers $1\leq r_0,r_1\leq (\log X)^{C_1}$ and any fixed $1\leq c_0,c_1\leq (\log X)^{C_2}$:
\[
\mathop{\sum\sum}_{\substack{\|n_0c_0,n_1c_1\|\leq X\\ \mathrm{gcd}(n_i,r_i)=1\;\forall i\in\{0,1\}\\ n_i\equiv q_i\bmod{8}\;\forall i\in\{0,1\}}} \hspace{-6pt}\frac{1}{\|n_0c_0,n_1c_1\|^2\tau(n_0)\tau(n_1)} = \frac{\mathfrak{S}_1(r_0,r_1)\log\log X}{c_0c_1} + O_{C_1,C_2}\left(\frac{\tau(r_0)\tau(r_1)\sqrt{\log \log X}}{c_0c_1}\right),
\]
where the implied constant depends only on $C_1$ and $C_2$ and for any odd $r_0$, $r_1$ we have 
\[
\mathfrak{S}_1(r_0,r_1) = \frac{2f_0^2}{\phi(8)^2\left(\prod_{p|2r_0}f_p\right)\left(\prod_{p|2r_1} f_p\right)}
\]
and $f_0$, $f_p$ are as defined in \eqref{TheConstants}.
\end{lemma}

\begin{remark}
    Note that the presence of $c_0$ and $c_1$ in the denominator of this asymptotic is due to their presence in the denominator of the summand and not due to their presence in the range of the $n_i$. It will be seen in the proof that they become untangled from the maximum.
\end{remark}

\begin{proof}
Set the sum to be $H(X)$. We split it into three regions depending on the value of $\|n_0c_0,n_1c_1\|$:
\[
H(X) = H_0(X)+H_1(X)-H_2(X)
\]
where
\[
H_0(X) = \sum_{\substack{n_0c_0\leq X\\ \mathrm{gcd}(n_0,r_0)=1\\ n_0\equiv q_0\bmod{8}}}\sum_{\substack{n_1\leq n_0c_0/c_1\\ \mathrm{gcd}(n_1,r_1)=1\\ n_1\equiv q_1\bmod{8}}} \frac{1}{n_0^2c_0^2 \tau(n_0)\tau(n_1)},\;\;\;
H_1(X) = \sum_{\substack{n_1c_1\leq X\\ \mathrm{gcd}(n_1,r_1)=1\\ n_1\equiv q_1\bmod{8}}}\sum_{\substack{n_0\leq n_1c_1/c_0\\ \mathrm{gcd}(n_0,r_0)=1\\ n_0\equiv q_0\bmod{8}}} \frac{1}{n_1^2c_1^2 \tau(n_0)\tau(n_1)},
\]
and
\[
H_2(X) = \sum_{\substack{\|n_0c_0,n_1c_1\|\leq X\\ \mathrm{gcd}(n_i,r_i)=1\;\forall i\in\{0,1\}\\ n_i\equiv q_i\bmod{8}\;\forall i\in\{0,1\}\\n_0c_0=n_1c_1}} \frac{1}{n_0^2c_0^2 \tau(n_0)\tau(n_1)}.
\]

Let us first consider $H_0(X)$. In order to use Lemma \ref{Siegel--Walfisz1} on the inner sum we need to ensure that $n_0c_0/c_1\geq 2$. We write:
\begin{align*}
H_0(X) &= \sum_{\substack{2c_1 \leq n_0c_0\leq X\\ \mathrm{gcd}(n_0,r_0)=1\\ n_0\equiv q_0\bmod{8}}}\sum_{\substack{n_1\leq n_0c_0/c_1\\ \mathrm{gcd}(n_1,r_1)=1\\ n_1\equiv q_1\bmod{8}}} \frac{1}{n_0^2c_0^2 \tau(n_0)\tau(n_1)} + O\left(\sum_{\substack{n_0c_0< 2c_1}}\sum_{\substack{n_1\leq n_0c_0/c_1}} \frac{1}{n_0^2c_0^2 \tau(n_0)\tau(n_1)}\right),
\end{align*}
To deal  with this second sum we note that $n_1\leq 2$ and swap the order of summation. Then it becomes
\[
\ll \sum_{n_1\leq 2}\sum_{n_1c_1/c_0\leq n_0}\frac{1}{c_0^2n_0^2} \ll \frac{1}{c_0c_1}.
\]
Thus we are left with
\begin{align*}
H_0(X) &= \sum_{\substack{2c_1 \leq n_0c_0\leq X\\ \mathrm{gcd}(n_0,r_0)=1\\ n_0\equiv q_0\bmod{8}}}\sum_{\substack{n_1\leq n_0c_0/c_1\\ \mathrm{gcd}(n_1,r_1)=1\\ n_1\equiv q_1\bmod{8}}} \frac{1}{n_0^2c_0^2 \tau(n_0)\tau(n_1)} + O\left(\frac{1}{c_0c_1}\right).
\end{align*}
Now applying Lemma \ref{Siegel--Walfisz1} with $Q=1$ and $C=3/2$ to the sum over $n_1$ we obtain:
\begin{align*}
H_0(X) &= \frac{\mathfrak{S}_0(r_1)}{\phi(8)c_0c_1}M(X) + O\left(\frac{\tau(r_1)(\log\log 3r_1)^{3/2}}{c_0c_1}E(X)\right),
\end{align*}
where
\begin{align*}
M(X) = \sum_{\substack{2c_1 \leq n_0c_0\leq X\\ \mathrm{gcd}(n_0,r_0)=1\\ n_0\equiv q_0\bmod{8}}} \frac{1}{n_0\tau(n_0)\sqrt{\log n_0c_0/c_1}}\;\;\textrm{and}\;\;
E(X) = \sum_{\substack{2c_1 \leq n_0c_0\leq X}} \frac{1}{n_0\tau(n_0)(\log n_0c_0/c_1)^{3/2}}
\end{align*}
Note that for small $n_0$ the error terms are roughly the same size as the main term; however, upon summing the $n_0$ over a large range, the dominance of the main term is maintained. Since $(\log\log r_1)^{3/2}\ll_{C_1} (\log\log\log X)^{3/2}$ it suffices to show that $E(X)\ll 1$. To see this, apply partial summation and Lemma \ref{Siegel--Walfisz1}:
\begin{align*}
    E(X) \ll& \frac{c_0}{X(\log X/c_1)^{3/2}}\sum_{\substack{2c_1 \leq n_0c_0\leq X}} \frac{1}{\tau(n_0)} + 1 + \int_{2c_1/c_0}^{X/c_0} \frac{1}{t^2(\log tc_0/c_1)^{3/2}}\sum_{\substack{2c_1/c_0 \leq n_0\leq t}} \frac{1}{\tau(n_0)} dt,\\
    \ll&\; 1 + \int_{\|2,2c_1/c_0\|}^{X/c_0} \frac{1}{t(\log tc_0/c_1)^{3/2}(\log t)^{1/2}} dt \\
    &+ \mathds{1}(c_1/c_0< 1)\int_{2c_1/c_0}^{2} \frac{1}{t^2(\log tc_0/c_1)^{3/2}}\sum_{\substack{2c_1/c_0 \leq n_0\leq t}} \frac{1}{\tau(n_0)} dt\\
    \ll&\; 1 + \int_{\|2,2c_1/c_0\|}^{X/c_0} \frac{1}{t(\log tc_0/c_1)^{3/2}(\log t)^{1/2}} dt \ll 1,
\end{align*}
where in the leading term of the second step we have used the fact that $c_1\leq (\log X)^{C_2}$ and a trivial bound to assert that
\[
\frac{1}{(\log X/c_1)^{3/2}} \ll 1.
\]
To see that the final integral converges, use the fact that $\sqrt{\log t}\geq \sqrt{\log 2}$ in this interval and we use the linear substitution $y=tc_0/c_1$. For $M(X)$ we increase the lower bound of this sums range:
\begin{align*}
M(X) = \sum_{\substack{(\log X)^{2C_2} \leq n_0\leq X/c_0\\ \mathrm{gcd}(n_0,r_0)=1\\ n_0\equiv q_0\bmod{8}}} \frac{1}{n_0\tau(n_0)\sqrt{\log n_0c_0/c_1}} + O\left(\sum_{\substack{n_0\leq (\log X)^{2C_2}}} \frac{1}{n_0\tau(n_0)}\right).
\end{align*}
A straightforward partial summation argument shows that this error term is $\ll \sqrt{\log\log X}$. For the other range we note that since $n_0\geq (c_0/c_1)^2$, we may use the Taylor series expansion
\[
\frac{1}{\sqrt{\log n_0c_0/c_1}} = \frac{1}{\sqrt{\log n_0}} + O\left(\frac{(\log c_0/c_1)}{(\log n_0)^{3/2}}\right).
\]
We then have
\[
M(X) = \sum_{\substack{(\log X)^{2C_2} \leq n_0\leq X/c_0\\ \mathrm{gcd}(n_0,r_0)=1\\ n_0\equiv q_0\bmod{8}}} \frac{1}{n_0\tau(n_0)\sqrt{\log n_0}} + O\left(\sum_{\substack{(\log X)^{2C_2} \leq n_0\leq X}} \frac{(\log c_0/c_1)}{n_0(\log n_0)^{3/2}}\right) + O_{C_2}\left(\sqrt{\log\log X}\right)
\]
The central error term sum converges and so,
\[
\sum_{\substack{(\log X)^{2C_2} \leq n_0\leq X/c_0}} \frac{(\log c_0/c_1)}{n_0(\log n_0)^{3/2}} \ll_{C_2} \frac{(\log c_0/c_1)}{\sqrt{\log\log X}} \ll_{C_2} \sqrt{\log\log X},
\]
For the leading sum in $M(X)$ we use partial summation to obtain
\begin{align*}
     & \int_{(\log X)^{2C_2}}^{X/c_0}\frac{1}{t^2\sqrt{\log t}}\sum_{\substack{n_0\leq t\\ \mathrm{gcd}(n_0,r_0)=1\\ n_0\equiv q_0\bmod{8}}} \frac{1}{\tau(n_0)}dt + O\left(\frac{1}{X\sqrt{\log X}}\sum_{n_0\leq X}\frac{1}{\tau(n_0)}\right)\\ &+ O\left(\int_{2}^{X}\frac{1}{t^2(\log t)^{3/2}}\sum_{\substack{n_0\leq t}} \frac{1}{\tau(n_0)}dt\right).
\end{align*}
Using the trivial bound for the sums in the error terms it is clear that they are $O(1)$. Finally we may apply Lemma \ref{Siegel--Walfisz1} with $Q=1$ and $C=3$ to the sum inside the main term, this time maintaining the constant:
\begin{align*}
    M(X) =& \frac{\mathfrak{S}_0(r_0)}{\phi(8)}\int_{(\log X)^{2C_2}}^{X/c_0}\frac{1}{t(\log t)}dt + O\left(\int_{2}^{X}\frac{(\log \log 3r_0)^{3/2}}{t(\log t)^2}dt\right) \\ &+ O\left(\int_{2}^{X}\frac{\tau(r_0)}{t(\log t)^{7/2}}dt\right) + O_{C_2}\left(\sqrt{\log \log X}\right)
\end{align*}
These latter integrals will converge as $X$ tends to infinity using $r_0\leq (\log X)^{C_1}$ to deal with the presence of $r_0$. The integral in the main term is
\[
\int_{(\log X)^{2C_2}}^{X/c_0}\frac{1}{t(\log t)}dt = (\log \log X/c_0) + O_{C_2}(\log \log \log X) = (\log \log X) + O_{C_2}(\log \log \log X)
\]
so that
\[
M(X) = \frac{\mathfrak{S}_0(r_0)}{\phi(8)}\log \log X + O_{C_2}\left(\sqrt{\log \log X}\right).
\]
Substituting our expressions for $M(X)$ and $E(X)$ into $H_0(X)$:
\begin{align*}
    H_0(X) &= \frac{\mathfrak{S}_0(r_0)\mathfrak{S}_0(r_1)}{\phi(8)^2c_0c_1}(\log \log X) + O_{C_1,C_2}\left(\frac{\sqrt{\log \log X}}{c_0c_1}+\frac{\tau(r_0)\tau(r_1)(\log\log\log X)^{3/2}}{c_0c_1}\right).
\end{align*}
We will similarly obtain the same expression for $H_1(X)$:
\begin{align*}
    H_1(X) &= \frac{\mathfrak{S}_0(r_0)\mathfrak{S}_0(r_1)}{\phi(8)^2c_0c_1}(\log \log X) + O_{C_1,C_2}\left(\frac{\sqrt{\log \log X}}{c_0c_1}+\frac{\tau(r_0)\tau(r_1)(\log\log\log X)^{3/2}}{c_0c_1}\right).
\end{align*}
For $H_2(X)$, suppose without loss in generality that $c_0\geq c_1$. Then
\begin{align*}
    H_2(X) \ll \frac{1}{c_0^2}\sum_{n_0\leq X/c_0} \frac{1}{n_0^2} \ll \frac{1}{c_0c_1}.
\end{align*}
Since $\mathfrak{S}_1(r_0,r_1) = \frac{2\mathfrak{S}_0(r_0)\mathfrak{S}_0(r_1)}{\phi(8)^2}$,
we are done.
\end{proof}

Using the same methods we can obtain the following variation:

\begin{lemma}\label{maintermlemma2}
    Let $X\geq 3$, $C_1,C_2>0$. Then for any fixed $1\leq c_0,c_1\leq (\log X)^{C_1}$:
\[
\mathop{\sum\sum}_{\substack{\|n_0c_0,n_1c_1\|\leq X \\ \|n_0d_0,n_1d_1\|>(\log X)^{C_2}}} \frac{\log \|n_0c_0,n_1c_1\|}{\|n_0c_0,n_1c_1\|^2\tau(n_0)\tau(n_1)} \ll_{C_1} \frac{\log X}{c_0c_1}
\]
where the implied constant depends only on $C_1,C_2>0$.
\end{lemma}

\begin{remark}
    By being more careful in the following proof we may also obtain the asymptotic
    \[
    \mathop{\sum\sum}_{\substack{\|n_0c_0,n_1c_1\|\leq X\\ \mathrm{gcd}(n_i,2r_i)=1\;\forall i\in\{0,1\}\\ n_i\equiv q_i\bmod{8}}} \frac{\log \|n_0c_0,n_1c_1\|}{\|n_0c_0,n_1c_1\|^2\tau(n_0)\tau(n_1)} \sim \frac{\mathfrak{S}_1(r_0,r_1)\log X}{\phi(8)^2 c_0c_1},
    \]
    but this is not needed later.
\end{remark}

\begin{proof}
Call the sum on the left-hand side $H(X)$. Then as before we may write
\[
H(X) \leq H_0(X)+H_1(X)
\]
where
\[
H_0(X) = \sum_{\substack{n_0c_0\leq X}}\sum_{\substack{n_1\leq n_0c_0/c_1}} \frac{(\log n_0c_0)}{n_0^2c_0^2 \tau(n_0)\tau(n_1)}\;\;\textrm{and}\;\;
H_1(X) = \sum_{\substack{n_1c_1\leq X}}\sum_{\substack{n_0\leq n_1c_1/c_0}} \frac{(\log n_1c_1)}{n_1^2c_1^2 \tau(n_0)\tau(n_1)},
\]
Looking at $H_0(X)$ first as in the previous proof we once more ensure the range over $n_1$ is $\geq 2$ in order to apply Lemma \ref{Siegel--Walfisz1}:
In order to use Lemma \ref{Siegel--Walfisz1} on the inner sum we need to ensure that $n_0c_0/c_1\geq 2$. We write:
\begin{align*}
H_0(X) &= \sum_{\substack{2c_1 \leq n_0c_0\leq X}}\sum_{\substack{n_1\leq n_0c_0/c_1}} \frac{(\log n_0c_0)}{n_0^2c_0^2 \tau(n_0)\tau(n_1)} + \sum_{\substack{n_0c_0< 2c_1\\ \mathrm{gcd}(n_0,r_0)=1\\ n_0\equiv q_0\bmod{8}}}\sum_{\substack{n_1\leq n_0c_0/c_1\\ \mathrm{gcd}(n_1,r_1)=1\\ n_1\equiv q_1\bmod{8}}} \frac{(\log n_0c_0)}{n_0^2c_0^2 \tau(n_0)\tau(n_1)}.\\
&= \sum_{\substack{2c_1 \leq n_0c_0\leq X\\ \mathrm{gcd}(n_0,r_0)=1\\ n_0\equiv q_0\bmod{8}}}\sum_{\substack{n_1\leq n_0c_0/c_1\\ \mathrm{gcd}(n_1,r_1)=1\\ n_1\equiv q_1\bmod{8}}} \frac{(\log n_0c_0)}{n_0^2c_0^2 \tau(n_0)\tau(n_1)} + O_{C_1}\left(\frac{(\log\log X)}{c_0c_1}\right).
\end{align*}
Now applying Lemma \ref{Siegel--Walfisz1} as an upper bound to the inner sum we obtain:
\begin{align*}
H_0(X) &\ll \frac{1}{c_0c_1}\sum_{\substack{2c_1 \leq n_0c_0\leq X}} \frac{(\log n_0c_0)}{n_0\tau(n_0)\sqrt{\log n_0c_0/c_1}}.
\end{align*}
We split this sum into two:
\begin{align*}
H_0(X) &\ll_{C_2} \frac{1}{c_0c_1}\sum_{\substack{(\log X)^{2C_2} \leq n_0\leq X/c_0}} \frac{(\log n_0c_0)}{n_0\tau(n_0)\sqrt{\log n_0c_0/c_1}} + \frac{1}{c_0c_1}\sum_{\substack{ n_0\leq (\log X)^{2C_2}}} \frac{(\log \log X)}{n_0\tau(n_0)}.
\end{align*}
The second sum here may be seen to be
\[
\frac{1}{c_0c_1}\sum_{\substack{ n_0\leq (\log X)^{2C_2}}} \frac{(\log \log X)}{n_0\tau(n_0)} \ll_{C_2} \frac{(\log\log X)^{3/4}}{c_0c_1}
\]
using a standard partial summation argument. For the first sum above we use a Taylor series expansion since we have $n_0\geq (\log X)^{2C_2}$. Thus
\[
\frac{1}{\sqrt{\log n_0c_0/c_1}} \ll_{C_2} \frac{1}{\sqrt{\log n_0}}.
\]
Using the logarithmic rule we may also get rid of the $c_0$ in the numerator of the summand. Overall we have,
\begin{align*}
H_0(X) \ll_{C_2}& \frac{1}{c_0c_1}\sum_{\substack{(\log X)^{2C_2} \leq n_0\leq X/c_0}} \frac{\sqrt{\log n_0}}{n_0\tau(n_0)} + \frac{\log\log X}{c_0c_1}\sum_{\substack{(\log X)^{2C_2} \leq n_0\leq X/c_0}} \frac{1}{n_0\tau(n_0)\sqrt{\log n_0}} \\&+ \frac{(\log \log X)^{3/4}}{c_0c_1}.
\end{align*}
Note that the second sum above is of the same form as the main term of $M(X)$ in Lemma \ref{maintermlemma1}, which we evaluated to be of order $\log\log X$. For the leading term above we use partial summation and Lemma \ref{Siegel--Walfisz1}:
\begin{align*}
\frac{1}{c_0c_1}\sum_{\substack{(\log X)^{2C_2} \leq n_0\leq X/c_0}} \frac{\sqrt{\log n_0}}{n_0\tau(n_0)} &\ll_{C_2} \frac{\sqrt{\log X}}{c_0c_1X}\sum_{\substack{n_0\leq X}} \frac{1}{\tau(n_0)} + \int_{2}^{X}\frac{\sqrt{\log t}}{c_0c_1t^2}\sum_{\substack{n_0\leq t}} \frac{1}{\tau(n_0)} dt\\
&\ll_{C_2} \frac{1}{c_0c_1} + \frac{1}{c_0c_1}\int_{2}^{X}\frac{dt}{t} \ll_{C_2} \frac{\log X}{c_0c_1}.
\end{align*}
Note that in the above computation, we bounded the derivative of $\frac{\sqrt{\log t}}{t}$ for brevity. Thus:
\[
H_0(X) \ll_{C_2} \frac{\log X}{c_0c_1},
\]
and we may similarly obtain the same bound for $H_1(X)$.
\end{proof}

The following result regards similar sums to the above, but over shorter ranges. We will see that they have similar flavour to sums already seen in the above proofs:

\begin{lemma}\label{maintermlemma3}
    Let $X\geq 3$, $C_1,C_2>0$. Then for any fixed $1\leq c_0,c_1\leq X^{C_1}$, $1\leq d_0,d_1\leq X^{C_2/2}$:
\[
\mathop{\sum\sum}_{\substack{\|n_0d_0,n_1d_1\|\leq X^{C_2}}} \frac{1}{\|n_0c_0,n_1c_1\|^2\tau(n_0)\tau(n_1)} \ll_{C_1,C_2} \left(\frac{\sqrt{ \log X}}{c_0c_1}\right),
\]
where the implied constant only depends on $C_1$ and $C_2$.
\end{lemma}

\begin{remark}
    The key point to note here is that the constants $c_0$ and $c_1$ are not included in the ranges for $n_0$ and $n_1$. This will cause trouble in the unwrapping argument before, especially since the constants may be larger than either variable very often. Instead it is enough just to use trivial bounds.
\end{remark}

\begin{proof}
The sum is at most $H_0(X)+H_1(X)$ where
\[
H_0(X) = \sum_{\substack{n_0\leq X^{C_2}}}\sum_{\substack{n_1\leq n_0c_0/c_1}} \frac{1}{n_0^2c_0^2 \tau(n_0)\tau(n_1)}\;\;\textrm{and}\;\;
H_1(X) = \sum_{\substack{n_1\leq X^{C_2}}}\sum_{\substack{n_0\leq n_1c_1/c_0}} \frac{1}{n_1^2c_1^2 \tau(n_0)\tau(n_1)}.
\]
Here we can use a trivial bound for the inner sums, giving
\[
H_0(X),H_1(X) \ll \sum_{\substack{n\leq X^{C_2}}} \frac{1}{nc_0c_1 \tau(n)}.
\]
Upon using partial summation and Lemma \ref{Siegel--Walfisz1} we obtain the desired bound.
\end{proof}

Next we put these together to obtain the following general main term result:

\begin{proposition}\label{maintermproposition}
    Let $X\geq 3$, $C_1,C_2,C_3>0$ and take any $\b{q}\in(\Z/8\Z)^{*4}$. Then for any fixed odd integers $1\leq r_0,r_1,r_2,r_3\leq(\log X)^{C_1}$ and fixed integers $1\leq c_0,c_1,c_2,c_3\leq (\log X)^{C_2}$, $1\leq d_0,d_1,d_2,d_3\leq (\log X)^{C_3/2}$ we have
    \begin{align*} \mathop{\sum\sum\sum\sum}_{\substack{\|n_0c_0,n_1c_1\|\cdot\|n_2c_2,n_3c_3\|\leq X\\\|n_0d_0,n_1d_1\|,\|n_2d_2,n_3d_3\|>(\log X)^{C_3} \\ \mathrm{gcd}(n_i,r_i)=1\;\forall\;0\leq i\leq 3\\ n_i\equiv q_i\bmod{8}\;\forall\;0\leq i\leq 3}} \hspace{-10pt}\frac{1}{\tau(n_0)\tau(n_1)\tau(n_2)\tau(n_3)} &= \frac{\mathfrak{S}_2(\b{r})X^2\log\log X}{c_0c_1c_2c_3\log X} \\ &+ O_{C_1,C_2,C_3}\left(\frac{\tau(r_0)\tau(r_1)\tau(r_2)\tau(r_3)X^2\sqrt{\log\log X}}{c_0c_1c_2c_3\log X}\right)
    \end{align*}
    where the implied constant depends at most on $C_1,C_2,C_3$ and we define
    \[
    \mathfrak{S}_2(\b{r}) = \frac{4f_0^4}{\phi(8)^4\left(\prod_{p|2r_0}f_p\right)\left(\prod_{p|2r_1}f_p\right)\left(\prod_{p|2r_2}f_p\right)\left(\prod_{p|2r_3}f_p\right)}.
    \]
\end{proposition}

\begin{proof}
    Call the sum on the left hand side $H(X)$. Then using Lemma \ref{hyperbolamethod} we may write
    \[
    H(X) = H_0(X)+H_1(X)+O(H_2(X))
    \]
    where
    \[
    H_0(X) = \mathop{\sum\sum\sum\sum}_{\substack{\|n_0c_0,n_1c_1\|\leq X^{1/2}\\ \|n_2c_2,n_3c_3\|\leq X/\|n_0c_0,n_1c_1\| \\ \|n_0d_0,n_1d_1\|,\|n_2d_2,n_3d_3\|>(\log X)^{C_3} \\ \mathrm{gcd}(n_i,2r_i)=1\;\forall\;0\leq i\leq 3\\ n_i\equiv q_i\bmod{8}\;\forall\;0\leq i\leq 3}}\frac{1}{\tau(n_0)\tau(n_1)\tau(n_2)\tau(n_3)},
    \]
    \[
    H_1(X) = \mathop{\sum\sum\sum\sum}_{\substack{\|n_2c_2,n_3c_3\|\leq X^{1/2}\\ \|n_0c_0,n_1c_1\|\leq X/\|n_2c_2,n_3c_3\| \\ \|n_0d_0,n_1d_1\|,\|n_2d_2,n_3d_3\|>(\log X)^{C_3} \\ \mathrm{gcd}(n_i,2r_i)=1\;\forall\;0\leq i\leq 3\\ n_i\equiv q_i\bmod{8}\;\forall\;0\leq i\leq 3}}\frac{1}{\tau(n_0)\tau(n_1)\tau(n_2)\tau(n_3)},
    \]
    and
    \[
    H_2(X) =  \mathop{\sum\sum\sum\sum}_{\substack{\|n_0c_0,n_1c_1\|\leq X^{1/2}\\ \|n_2c_2,n_3c_3\|\leq X^{1/2}}}\frac{1}{\tau(n_0)\tau(n_1)\tau(n_2)\tau(n_3)}.
    \]
    Let us first deal with $H_2(X)$. We write
    \begin{align*}
    H_2(X) \ll \prod_{i=0}^{3}\left(\sum_{\substack{n_ic_i\leq X^{1/2}}}\frac{1}{\tau(n_i)}\right) \ll \frac{X^2}{c_0c_1c_2c_3(\log X)^2},
    \end{align*}
    by Lemma \ref{Siegel--Walfisz1}. Now let us consider $H_0(X)$. Here we may add in the terms for which $\|n_2d_2,n_3d_3\|\leq (\log X)^{C_3}$ at the cost of a negligible error term since,
    \begin{equation}\label{section3smallsum}
    \mathop{\sum\sum\sum\sum}_{\substack{\|n_0c_0,n_1c_1\|\leq X^{1/2} \\ \|n_2d_2,n_3d_3\|\leq (\log X)^{C_3}}} \frac{1}{\tau(n_0)\tau(n_1)\tau(n_2)\tau(n_3)} \ll X(\log X)^{2C_3}.
    \end{equation}
    Then
    \[
    H_0(X) = \mathop{\sum\sum\sum\sum}_{\substack{\|n_0c_0,n_1c_1\|\leq X^{1/2}\\ \|n_2c_2,n_3c_3\|\leq X/\|n_0c_0,n_1c_1\|\\ \|n_0d_0,n_1d_1\|>(\log X)^{C_3} \\ \mathrm{gcd}(n_i,2r_i)=1\;\forall\;0\leq i\leq 3\\ n_i\equiv q_i\bmod{8}\;\forall\;0\leq i\leq 3}}\frac{1}{\tau(n_0)\tau(n_1)\tau(n_2)\tau(n_3)} + O(X(\log X)^{2C_3}).
    \]
    Using Lemma \ref{Siegel--Walfisz1} the sum over $n_2$ and $n_3$ is
    \begin{align*}
        \frac{\mathfrak{S}_0(r_2)\mathfrak{S}_0(r_3)X^2}{\phi(8)^2c_2c_3\|n_0c_0,n_1c_1\|^2(\log(X/\|n_0c_0,n_1c_1\|))} + O\left(\frac{X^2(\log\log \|3r_2,3r_3\|)^{3/2}}{c_2c_3\|n_0c_0,n_1c_1\|^2(\log (X/\|n_0c_0,n_1c_1\|))^{2}}\right).
    \end{align*}
    Note that we have suppressed the arbitrary log saving error term in this calculation. This can be done by noting that $Q=1$ and $\tau(r_i)\ll_{C_1} r_i^{1/C_1}\ll_{C_1} (\log X)$ so that
    \[
    \frac{\tau(r_i)X^2}{(\log X/\|n_0c_0,n_1c_1\|)^{C_4}} \ll \frac{X^2(\log\log \|3r_2,3r_3\|)^{3/2}}{(\log (X/\|n_0c_0,n_1c_1\|))^{2}}
    \]
    for $C_4$ chosen sufficiently large. Therefore we have
    \[
    H_0(X) = \frac{\mathfrak{S}_0(r_2)\mathfrak{S}_0(r_3)X^2}{c_2c_3}M_{0}(X)+O\left(\frac{X^2}{c_2c_3}E_{0}(X)\right) + O(X(\log X)^{2C_3})
    \]
    where
    \[
    M_0(X) = \mathop{\sum\sum}_{\substack{\|n_0c_0,n_1c_1\|\leq X^{1/2}\\ \|n_0d_0,n_1d_1\|>(\log X)^{C_3} \\ \mathrm{gcd}(n_i,2r_i)=1\;\forall\;0\leq i\leq 1\\ n_i\equiv q_i\bmod{8}\;\forall\;0\leq i\leq 1}}\frac{1}{\|n_0c_0,n_1c_1\|^2\tau(n_0)\tau(n_1)(\log X/\|n_0c_0,n_1c_1\|)}
    \]
    and
    \[
    E_0(X) = \mathop{\sum\sum}_{\substack{\|n_0c_0,n_1c_1\|\leq X^{1/2}}}\frac{(\log\log\|3r_2,3r_3\|)^{3/2}}{\|n_0c_0,n_1c_1\|^2\tau(n_0)\tau(n_1)(\log X/\|n_0c_0,n_1c_1\|)^2}.
    \]
    Using $\|n_0c_0,n_1c_1\|\leq X^{1/2}$, the usual Taylor series manoeuvre and Lemma \ref{maintermlemma1} we have
    \begin{align*}
    E_0(X) &\ll \frac{1}{(\log X)^2}\mathop{\sum\sum}_{\substack{\|n_0c_0,n_1c_1\|\leq X^{1/2}}}\frac{(\log\log\|3r_2,3r_3\|)^{3/2}}{\|n_0c_0,n_1c_1\|^2\tau(n_0)\tau(n_1)}\\
    &\ll \frac{(\log\log X)(\log\log\|3r_2,3r_3\|)^{3/2}}{c_0c_1(\log X)^2}
    \end{align*}
    which is sufficient. Now we turn to $M_0(X)$. Since $\|n_0c_0,n_1c_1\|\leq \sqrt{X}$ we may use a geometric series argument to write
    \[
    \frac{1}{\log X/\|n_0c_0,n_1c_1\|} = \frac{1}{\log X}+O\left(\frac{\log \|n_0c_0,n_1c_1\|}{(\log X)^2}\right).
    \]
    We have from Lemma \ref{maintermlemma1} (with Lemma \ref{maintermlemma3} to add in the terms for which $\|n_0d_0,n_1d_1\|\leq (\log X)^{C_3}$),
    \begin{align*}
    \mathop{\sum\sum}_{\substack{\|n_0c_0,n_1c_1\|\leq X^{1/2}\\ \|n_0d_0,n_1d_1\|>(\log X)^{C_3} \\ \mathrm{gcd}(n_i,2r_i)=1\;\forall\;0\leq i\leq 1\\ n_i\equiv q_i\bmod{8}\;\forall\;0\leq i\leq 1}}\hspace{-3pt}\frac{1}{\|n_0c_0,n_1c_1\|^2\tau(n_0)\tau(n_1)} &= \frac{\mathfrak{S}_1(r_0,r_1)\log\log (X^{1/2})}{c_0c_1} \\ &+ O_{C_1,C_2,C_3}\left(\frac{\tau(r_0)\tau(r_1)\sqrt{\log\log (X^{1/2})}}{c_0c_1}\right)
    \end{align*}
    and from Lemma \ref{maintermlemma2} that
    \[
    \mathop{\sum\sum}_{\substack{\|n_0c_0,n_1c_1\|\leq X^{1/2}\\ \|n_0d_0,n_1d_1\|>(\log X)^{C_3} \\ \mathrm{gcd}(n_i,2r_i)=1\;\forall\;0\leq i\leq 1\\ n_i\equiv q_i\bmod{8}\;\forall\;0\leq i\leq 1}}\frac{\log \|n_0c_0,n_1c_1\|}{\|n_0c_0,n_1c_1\|^2\tau(n_0)\tau(n_1)} \ll_{C_1,C_2,C_3} \frac{\log X}{c_0c_1}.
    \]
    We therefore conclude that
    \[
    M_0(X) = \frac{\mathfrak{S}_1(r_0,r_1)\log\log (X^{1/2})}{c_0c_1} + O_{C_1,C_2,C_3}\left(\frac{\sqrt{\log\log X}}{c_0c_1}\right).
    \]
    Note that when using the Taylor series above, the error term may in fact be of the same order as the main term. In writing it this way we are in fact splitting a constant into two parts - one independant of the $n_i$, which contributes to the $(\log\log X)$ by Lemma \ref{maintermlemma1} and a part dependent on the $n_i$ which contributes to an error of $O(1)$ by Lemma \ref{maintermlemma2}. We have now shown
    \begin{align*}
    H_0(X) &= \frac{ \mathfrak{S}_0(r_2)\mathfrak{S}_0(r_3)\mathfrak{S}_1(r_0,r_1)X^2\log\log X}{\phi(8)^2c_0c_1c_2c_3(\log X)} + O_{C_1,C_2,C_3}\left(\frac{\tau(r_0)\tau(r_1)X^2\sqrt{\log\log X}}{c_0c_1c_2c_3(\log X)}\right).
    \end{align*}
    Evaluating $H_1(X)$ in the same way we obtain the same result with $r_0$ and $r_1$ switched with $r_2$ and $r_3$. Noting that $\frac{\mathfrak{S}_0(r_0)\mathfrak{S}_0(r_1)\mathfrak{S}_1(r_2,r_3)}{\phi(8)^2} = \frac{\mathfrak{S}_0(r_2)\mathfrak{S}_0(r_3)\mathfrak{S}_1(r_0,r_1)}{\phi(8)^2} = \frac{\mathfrak{S}_2(\b{r})}{2}$ we combine the expressions for $H_0(X)$, $H_1(X)$ and $H_2(X)$ to conclude the proof.
\end{proof}

\subsection{Small Conductors - Symmetric Hyperbola Method}
This is the easiest of the three cases: all that is required for us to do is to apply Lemmas \ref{hyperbolamethod} and \ref{Siegel--Walfisz1} appropriately. To save space we introduce the following summation conditions:

\begin{equation}\label{section3sumconds}
    \begin{cases}
        \|n_0d_0,n_1d_1\|,\|n_2d_2,n_3d_3\|>(\log X)^{D}\\
        \mathrm{gcd}(n_i,r_i)=1\;\forall\; 0\leq i\leq 3\\
        n_i\equiv q_i\bmod{8}\;\forall\; 0\leq i\leq 3
    \end{cases}
\end{equation}
where $D>0$ and the $r_i$ and $d_i$ are some integers and $q_i\in(\Z/8\Z)^{*}$ for each $i$.
\begin{lemma}\label{symmetric-hyperbola-sums}
    Let $X\geq 3$, $C_1,C_2,C_3>0$ and fix odd integers $Q_0,Q_2$ and some $\b{q}\in(\Z/8\Z)^{*4}$. Suppose $\chi_0$ and $\chi_2$ are non-principal Dirichlet characters modulo $Q_0$ and $Q_2$ and that $g_1,g_3:\N\rightarrow\C$ are multiplicative functions such that $\lvert g_1(n)\rvert,\lvert g_3(n)\rvert \leq 1$ for all $n\in\N$. Then for any odd integers $1\leq r_0,r_1,r_2,r_3\leq(\log X)^{C_1}$ such that $\mathrm{gcd}(r_i,Q_i)$ for $i=0,2$ and any fixed $1\leq c_0,c_1,c_2,c_3\leq (\log X)^{C_2}$, $1\leq d_0,d_1,d_2,d_3\leq (\log X)^{C_3/2}$ we have
    \[ \mathop{\sum\sum\sum\sum}_{\substack{\b{n}\in\N^4,\|n_0c_0,n_1c_1\|\cdot\|n_2c_2,n_3c_3\|\leq X\\ \eqref{section3sumconds}}} \frac{\chi_0(n_0)g_1(n_1)\chi_2(n_2)g_3(n_3)}{\tau(n_0)\tau(n_1)\tau(n_2)\tau(n_3)} \ll_{C_1,C_2,C_3,C_4} \frac{Q_0Q_2 X^2}{c_0c_1c_2c_3(\log X)^{C_4}}
    \]
    for any $C_4>0$ where the implied constant depends at most on $C_1,C_2,C_3$ and $C_4$. Note we have used $D=C_3$ in \eqref{section3sumconds}.
\end{lemma}

\begin{proof}
Call this sum $H(X)$. Then using Lemma \ref{hyperbolamethod} we obtain
\[
H(X) = H_0(X) + H_1(X) - H_2(X),
\]
where
\begin{equation}\label{non-principalsymmetrichyperbolicregionsum0}
H_0(X) = \mathop{\sum\sum\sum\sum}_{\substack{\|n_0c_0,n_1c_1\|\leq X^{1/2}\\ \|n_2c_2,n_3c_3\|\leq X/\|n_0c_0,n_1c_1\| \\ \eqref{section3sumconds}}}\frac{\chi_0(n_0)g_1(n_1)\chi_2(n_2)g_3(n_3)}{\tau(n_0)\tau(n_1)\tau(n_2)\tau(n_3)},
\end{equation}

\begin{equation}
H_1(X) = \mathop{\sum\sum\sum\sum}_{\substack{\|n_2c_2,n_3c_3\|\leq X^{1/2}\\ \|n_0c_0,n_1c_1\|\leq X/\|n_2c_2,n_3c_3\| \\ \eqref{section3sumconds}}}\frac{\chi_0(n_0)g_1(n_1)\chi_2(n_2)g_3(n_3)}{\tau(n_0)\tau(n_1)\tau(n_2)\tau(n_3)},
\end{equation}
and
\begin{equation}\label{typicalrectangularregioncharactersum}
H_2(X) = \mathop{\sum\sum\sum\sum}_{\substack{\|n_0c_0,n_1c_1\|,\|n_2c_2,n_3c_3\|\leq X^{1/2} \\ \eqref{section3sumconds}}}\frac{\chi_0(n_0)g_1(n_1)\chi_2(n_2)g_3(n_3)}{\tau(n_0)\tau(n_1)\tau(n_2)\tau(n_3)}.
\end{equation}
We deal with $H_2(X)$ first. Using the trivial bound \eqref{section3smallsum}, we may add in the terms for which $\|n_0d_0,n_1d_1\|>(\log X)^{C_3}$ or $\|n_2d_2,n_3d_3\|>(\log X)^{C_3}$ at the cost of a small error term:
\begin{equation*}
H_2(X) = \mathop{\sum\sum\sum\sum}_{\substack{\|n_0c_0,n_1c_1\|,\|n_2c_2,n_3c_3\|\leq X^{1/2} \\ \mathrm{gcd}(n_i,r_i)=1\;\forall 0\leq i\leq 3\\ n_i\equiv q_i\bmod{8}\;\forall 0\leq i\leq 3}}\frac{\chi_0(n_0)g_1(n_1)\chi_2(n_2)g_3(n_3)}{\tau(n_0)\tau(n_1)\tau(n_2)\tau(n_3)} + O(X(\log X)^{2C_3}).
\end{equation*}
This sum is now separable so that, using trivial bounds for the sums over $n_1$ and $n_3$ and Lemma \ref{Siegel--Walfisz1} for the sums over $n_2$ and $n_4$ we obtain:
\begin{align*}
H_2(X) =& \left(\sum_{\substack{n_0c_0\leq X^{1/2}\\ \mathrm{gcd}(n_0,r_0)=1\\ n_0\equiv q_0\bmod{8}}}\frac{\chi_0(n_0)}{\tau(n_0)}\right)\left(\sum_{\substack{n_1c_1\leq X^{1/2}\\ \mathrm{gcd}(n_1,r_1)=1\\ n_1\equiv q_1\bmod{8}}}\frac{g_1(n_1)}{\tau(n_1)}\right)\left(\sum_{\substack{n_2c_2\leq X^{1/2}\\ \mathrm{gcd}(n_2,r_2)=1\\ n_2\equiv q_2\bmod{8}}}\frac{\chi_2(n_2)}{\tau(n_2)}\right)\left(\sum_{\substack{n_3c_3\leq X^{1/2}\\ \mathrm{gcd}(n_3,r_3)=1\\ n_3\equiv q_3\bmod{8}}}\frac{g_3(n_3)}{\tau(n_3)}\right)\\&+O(X(\log X)^{2C_2})\\
\ll&_{C_1,C_2,C_3,C_4}\;\left(\frac{\tau(r_0)Q_0X^{1/2}}{c_0(\log X)^{C_4}}\right)\left(\frac{X^{1/2}}{c_1}\right)\left(\frac{\tau(r_2)Q_2X^{1/2}}{c_2(\log X)^{C_4}}\right)\left(\frac{X^{1/2}}{c_3}\right)\\
\ll&_{C_1,C_2,C_3,C_4}\;\frac{Q_0Q_2X^2}{c_0c_1c_2c_3(\log X)^{C_4}}.
\end{align*}
Note also that the assumption $r_i\leq (\log X)^{C_1}$ we may ignore the $\tau(r_i)$ upon choosing $C_4$ appropriately large, and since $c_i\leq (\log X)^{C_2}$ for all $i$ we may use the bound
\[
\frac{1}{(\log X/c_i)^{C_4}}\ll_{C_2,C_4}\frac{1}{(\log X)^{C'_4}}.
\]
We will use these remarks again without mentioning. Next we deal with $H_0(X)$. Add in the terms where $\|n_2d_2,n_3d_3\|\leq (\log X)^{C_3}$ at the cost of an error term of size $O(X(\log X)^{2C_3})$ by using a trivial bound. Then, performing the sum over $n_2$ and $n_3$ first we write,
\begin{align*}
    H_0(X) \ll_{C_1,C_3} \mathop{\sum\sum}_{\substack{n_0c_0,n_1c_1\leq X^{1/2}}}\left\lvert\sum_{\substack{n_2c_2\leq X/\|c_0n_0,c_1n_1\|\\ \mathrm{gcd}(n_2,r_2)=1\\ n_2\equiv q_2\bmod{8}}}\frac{\chi_2(n_2)}{\tau(n_2)}\right\rvert\left\lvert\sum_{\substack{n_3c_3\leq X/\|c_0n_0,c_1n_1\|\\ \mathrm{gcd}(n_3,r_3)=1\\ n_3\equiv q_3\bmod{8}}}\frac{g_3(n_3)}{\tau(n_3)}\right\rvert+O(X(\log X)^{2C_3}).
\end{align*}
Using a trivial bound for the sum over $n_3$ and Lemma \ref{Siegel--Walfisz1} for non-trivial characters (noting that the range is relatively large since $\|n_0c_0,n_1c_1\|\leq X^{1/2}$), we obtain
\begin{align*}
    H_0(X) \ll_{C_1,C_2,C_3,C_4}\;&\; \frac{\tau(r_2)Q_2X^2}{c_2c_3(\log X)^{2C_4+2}}\sum_{n_0c_0,n_1c_1\leq X^{1/2}}\frac{1}{\|n_0c_0,n_1c_1\|^2}\\
    \ll_{C_1,C_2,C_3,C_4}\;&\; \frac{Q_2X^2}{c_0c_1c_2c_3(\log X)^{C_4}}
\end{align*}
where we have used the straightforward bound
\[
\sum_{n_0c_0,n_1c_1\leq X^{1/2}}\frac{1}{\|n_0c_0,n_1c_1\|^2} \ll \frac{(\log X^{1/2}/c_0)(\log X^{1/2}/c_1)}{c_0c_1} \ll \frac{(\log X)^{2}}{c_0c_1}.
\]
For $H_1(X)$ we use an identical argument to that of $H_0(X)$ with $n_0,n_1$ switching roles with $n_2,n_3$. This yields
\[
H_1(X) \ll_{C_1,C_2,C_3,C_4}\;\; \frac{Q_0X^2}{c_0c_1c_2c_3(\log X)^{C_4}}.
\]
Combining these three bounds gives the result.
\end{proof}

To conclude this section we want to average this result over a small range of conductors. For this purpose it will be necessary to specialise to the case where the characters are Jacobi symbols. For $m\in\N$ odd, let $\psi_{m}(\cdot)$ denote generically either the Jacobi symbol $\left(\frac{\cdot}{m}\right)$ or the Jacobi symbol $\left(\frac{m}{\cdot}\right)$.

\begin{proposition}\label{symmetrictypeaverage1}
    Let $X\geq 3$, $C_1,C_2,C_3>0$ and fix odd integers $Q_0,Q_2$ and some $\b{q}\in(\Z/8\Z)^{*4}$, $\Tilde{\b{q}}\in(\Z/8\Z)^{*4}$. Fixing some odd integers $1\leq r_0,r_1,r_2,r_3\leq(\log X)^{C_1}$ such that $\mathrm{gcd}(r_i,Q_i)=1$ for $i=0,2$ and any $1\leq c_0,c_1,c_2,c_3\leq (\log X)^{C_2}$, $1\leq d_0,d_1,d_2,d_3\leq (\log X)^{C_3/2}$ we define, for any $\b{m}\in\N^4$,
    \[
    H(X,\b{m}) = \mathop{\sum\sum\sum\sum}_{\substack{\b{n}\in\N^{4}, \|n_0m_0c_0,n_1m_1c_1\|\cdot\|n_2m_2c_2,n_3m_3c_3\|\leq X\\ \eqref{section3sumconds}}}\frac{\psi_{Q_0m_0m_1}(n_2n_3)\psi_{Q_2m_2m_3}(n_0n_1)}{\tau(n_0)\tau(n_1)\tau(n_2)\tau(n_3)},
    \]
    where we use \eqref{section3sumconds} with $D=C_3$. Then for any $C_4>0$:
    \begin{align*} \mathop{\sum\sum\sum\sum}_{\substack{\b{m}\in\N^{4},\|m_0,m_1\|,\|m_2,m_3\|\leq (\log X)^{C_3}\\ \mathrm{gcd}(m_0m_1,Q_0r_2r_3)=\mathrm{gcd}(m_2m_3,Q_1r_0r_1)=1\\ \b{m}\equiv \Tilde{\bar{q}}\bmod{8}\\ Q_0m_0m_1\;\text{and}\;Q_2m_2m_3\neq 1}}& \frac{\mu^2(2m_0m_1m_2m_3)\lvert H(X,\b{m})\rvert}{\tau(m_0)\tau(m_1)\tau(m_2)\tau(m_3)} \ll_{C_1,C_2,C_3,C_4}\frac{Q_0Q_2X^2}{c_0c_1c_2c_3(\log X)^{C_4}}.
    \end{align*}
\end{proposition}

\begin{proof}
    By the $\mathrm{gcd}$ conditions on $Q_i$ and the $m_i$, the condition that $Q_0m_0m_1$ and $Q_2m_2m_3 \neq 1$ and the term $\mu^2(m_0m_1m_2m_3)=1$, we know that for each $\b{m}$ considered in the average, the quadratic characters $\psi_{Q_0m_0m_1}$ and $\psi_{Q_1m_2m_3}$ are non-principal. Therefore, Lemma \ref{symmetric-hyperbola-sums} tells us that, for each $\b{m}$ considered in the average,
    \begin{align*}
    H(X,\b{m}) &\ll_{C_1,C_2,C_3,C_4} \frac{Q_0Q_2m_0m_1m_2m_3X^2}{m_0m_1m_2m_3c_0c_1c_2c_3(\log X)^{C_4+4C_3}}\\
    &\ll_{C_1,C_2,C_3,C_4} \frac{Q_0Q_2X^2}{c_0c_1c_2c_3(\log X)^{C_4+4C_3}}.
    \end{align*}
    Therefore, by summing over the given $\b{m}$ the average can be seen to be bounded by
    \begin{align*}    &\ll_{C_1,C_2,C_3,C_4}\frac{Q_0Q_2X^2(\log X)^{4C_3}}{c_0c_1c_2c_3(\log X)^{C_4+4C_3}}\ll_{C_1,C_2,C_3,C_4} \frac{Q_0Q_2X^2}{c_0c_1c_2c_3(\log X)^{C_4}}.
    \end{align*}
\end{proof}

\subsection{Small Conductors: Asymmetric Hyperbola Method}
We begin with a technical lemma similar in form to Lemma \ref{maintermlemma1}.

\begin{lemma}\label{non-symmetric-hyperbola-sums-mainterm1}
Let $X\geq 3$, $C_1,C_2>0$, $0<\epsilon<1$ and define $Y=\exp((\log X)^{\epsilon})$. Fix some odd integers $Q_0$, $Q_1$ and some $\b{q}\in(\Z/8\Z)^{*2}$ and suppose $\chi_0$ and $\chi_1$ are non-principal characters modulo $Q_0$ and $Q_1$ respectively. Then for any odd integers $1\leq r_0,r_1\leq(\log X)^{C_1}$ and any fixed integers $1\leq c_0,c_1\leq (\log X)^{C_2/16}$, $1\leq d_0,d_1\leq (\log X)^{C_2/4}$ we have
\begin{equation*}
\mathop{\sum\sum}_{\substack{n_0c_0,n_1c_1\leq Y\\ \|n_0d_0,n_1d_1\|>(\log X)^{C_2}\\ \mathrm{gcd}(n_i,r_i)=1\;\forall 0\leq i\leq 1\\ n_i\equiv q_i\bmod{8}\;\forall 0\leq i\leq 1}}\frac{\chi_0(n_0)\chi_1(n_1)}{\|n_0c_0,n_1c_1\|^2\tau(n_0)\tau(n_1)} \ll_{C_1,C_2,C_3} \frac{\tau(r_0)\tau(r_1)(Q_0+Q_1)}{c_0c_1(\log\log X)^{C_3}}
\end{equation*}
for any $C_3>0$ where the implied constant depends at most on $C_1,C_2$ and $C_3$.
\end{lemma}

\begin{remark}
    The philosophy with this sum, as with many others like it that appear throughout this section, is that it should converge and so the lower bounds should yield some saving. To see this, note that the sums considered above are similar to
    \[
    \mathop{\sum\sum}_{\substack{n_0c_0,n_1c_1\leq Y\\ \mathrm{gcd}(n_i,r_i)=1\;\forall 0\leq i\leq 1\\ n_i\equiv q_i\bmod{8}\;\forall 0\leq i\leq 1}}\frac{\chi_0(n_0)\chi_1(n_1)}{n_0n_1c_0c_1\tau(n_0)\tau(n_1)}
    \]
    which is separable in each variable and is more readily seen to converge by comparing the two sums to the Dirichlet series
    \[
    D(1,\chi_i) = \sum_{n=1}^{\infty}\frac{\chi_i(n)}{n\tau(n)}.
    \]
    More will be said about this idea in the third part of \S \ref{hyperbolicanalysis2}.
\end{remark}

\begin{proof}
    Call the sum $H(X)$. The key difference to the above remark is that we run into some difficulty when trying to untangle the constants $c_0,d_0,c_1$ and $d_1$. Indeed we have to split both maximums simultaneously in order to obtain sums of a familiar form. We have four cases:
    \begin{itemize}
        \item[$(1)$] $n_0c_0 \geq n_1c_1$ and $n_0d_0 \geq n_1d_1$;
        \item[$(2)$] $n_1c_1 > n_0c_0$ and $n_1d_1 > n_0d_0$;
        \item[$(3)$] $n_0c_0 \geq n_1c_1$ and $n_1d_1 > n_0d_0$;
        \item[$(4)$] $n_1c_1 > n_0c_0$ and $n_0d_0 \geq n_1d_1$.
    \end{itemize}
    We note immediately that the conditions of $(3)$ imply that $\frac{d_1}{d_0}>\frac{c_1}{c_0}$
    while the conditions of $(4)$ imply $\frac{d_1}{d_0}<\frac{c_1}{c_0}$ and so only one of them will apply. Without loss in generality, we will assume that $(3)$ case holds. Now we split $H(X)$ into regions in which $n_0$ and $n_1$ satisfy these conditions. We have
    \[
    H(X) = H_1(X) + H_2(X) + H_3(X).
    \]
    where
    \[
    H_1(X) = \mathop{\sum}_{\substack{(\log X)^{C_2}/d_0<n_0\leq Y/c_0\\ \mathrm{gcd}(n_0,r_0)=1\\ n_0\equiv q_0\bmod{8}}}\mathop{\sum}_{\substack{n_1\leq \min(n_0d_0/d_1,n_0c_0/c_1)\\ \mathrm{gcd}(n_1,r_1)=1\\ n_1\equiv q_1\bmod{8}}}\frac{\chi_0(n_0)\chi_1(n_1)}{n_0^2c_0^2\tau(n_0)\tau(n_1)},
    \]
    \[
    H_2(X) = \mathop{\sum}_{\substack{(\log X)^{C_2}/d_1<n_1\leq Y/c_1\\ \mathrm{gcd}(n_1,r_1)=1\\ n_1\equiv q_1\bmod{8}}}\mathop{\sum}_{\substack{n_0< \min(n_1d_1/d_0,n_1c_1/c_0)\\ \mathrm{gcd}(n_1,r_1)=1\\ n_1\equiv q_1\bmod{8}}}\frac{\chi_0(n_0)\chi_1(n_1)}{n_1^2c_1^2\tau(n_0)\tau(n_1)},
    \]
    and
    \[
    H_3(X) = \mathop{\sum}_{\substack{n_0\leq Y/c_0\\ \mathrm{gcd}(n_0,r_0)=1\\ n_0\equiv q_0\bmod{8}}}\mathop{\sum}_{\substack{\|(\log X)^{C_2}/d_1,n_0d_0/d_1\|<n_1\leq n_0c_0/c_1\\ \mathrm{gcd}(n_1,r_1)=1\\ n_1\equiv q_1\bmod{8}}}\frac{\chi_0(n_0)\chi_1(n_1)}{n_0^2c_0^2\tau(n_0)\tau(n_1)}.
    \]
    The sums $H_1(X)$ and $H_2(X)$ may be dealt with similarly. For $H_1(X)$ we note that in these regions $\min(\frac{n_0d_0}{d_1},\frac{n_0c_0}{c_1})>(\log X)^{3C_2/4}$, and so the range over $n_1$ increases with $X$. We therefore apply Lemma \ref{Siegel--Walfisz1} to it to obtain:
    \[
    H_1(X) \ll_{C_3} \mathop{\sum}_{\substack{(\log X)^{C_2}/d_0<n_0}} \frac{\tau(r_1)Q_1\min(d_0/d_1,c_0/c_1)}{n_0c_0^2(\log (n_0\min(d_0/d_1,c_0/c_1)))^{C_3+1}},
    \]
    for any $C_3>0$. Now we use $\min(d_0/d_1,c_0/c_1) \leq c_0/c_1$ to bound the numerator; for the denominator, we have that $n_0 > (\log X)^{C_2}/d_0 \geq (\log X)^{3C_2/4} \geq \min(d_0/d_1,c_0/c_1)^2$ so we may obtain the bound
    \[
    [\log (n_0\min(d_0/d_1,c_0/c_1))]^{-C_3-1} =\left[(\log n_0)\hspace{-2pt}\left(\hspace{-2pt}1+\frac{\log (\min(d_0/d_1,c_0/c_1))}{\log n_0}\right)\hspace{-2pt}\right]^{-C_3-1} \hspace{-10pt}\ll_{C_3} (\log n_0)^{-C_3-1}.
    \]
    Alternatively, one could bound the denominator using a Taylor series expansion. This will yield,
    \[
    H_1(X) \ll_{C_3} \mathop{\sum}_{\substack{(\log X)^{C_2}/d_0<n_0}} \frac{\tau(r_1)Q_1}{n_0c_0c_1(\log 
    n_0)^{C_3+1}}.
    \]
    Since $C_3>0$, this is the tail of a convergent series and so we obtain
    \[
    H_1(X) \ll_{C_2,C_3} \frac{\tau(r_1)Q_1}{c_0c_1(\log\log X)^{C_3}}.
    \]
    Similarly,
    \[
    H_2(X)_{C_2,C_3} \ll_{C_2,C_3}\frac{\tau(r_0)Q_0}{c_0c_1(\log\log X)^{C_3}}.
    \]
    We now turn to $H_3(X)$. First note that this sum is $0$ unless
    \[
    \frac{n_0c_0}{c_1}\geq \frac{(\log X)^{C_2}}{d_1}.
    \]
    We may therefore add this condition to the sum of $H_3(X)$ with no cost. This will lead to:
    \[
    H_3(X) = \mathop{\sum}_{\substack{c_1(\log X)^{C_2}/c_0d_1<n_0\leq Y/c_0\\ \mathrm{gcd}(n_0,r_0)=1\\ n_0\equiv q_0\bmod{8}}}\mathop{\sum}_{\substack{\|(\log X)^{C_2}/d_1,n_0d_0/d_1\|<n_1\leq n_0c_0/c_1\\ \mathrm{gcd}(n_1,r_1)=1\\ n_1\equiv q_1\bmod{8}}}\frac{\chi_0(n_0)\chi_1(n_1)}{n_0^2c_0^2\tau(n_0)\tau(n_1)}.
    \]
    In this case it is unclear whether or not the range over $n_1$ is guaranteed to increase with $X$. In order to deal with this we write the sum over $n_1$ as
    \[
    \mathop{\sum}_{\substack{n_1\leq n_0c_0/c_1\\ \mathrm{gcd}(n_1,r_1)=1\\ n_1\equiv q_1\bmod{8}}}\frac{\chi_1(n_1)}{\tau(n_1)} - \mathop{\sum}_{\substack{n_1\leq \|(\log X)^{C_2}/d_1,n_0d_0/d_1\|\\ \mathrm{gcd}(n_1,r_1)=1\\ n_1\equiv q_1\bmod{8}}}\frac{\chi_1(n_1)}{\tau(n_1)}.
    \]
    Both of these are sums with ranges which grow with $X$ and so we may apply Lemma \ref{Siegel--Walfisz1} to them. Upon doing this to the first sum above and then summing over $n_0$ we obtain,
    \begin{align*}
    \ll_{C_3} \mathop{\sum}_{\substack{(\log X)^{11C_2/16}<n_0\leq Y}}\frac{\tau(r_1)Q_1}{n_0c_0c_1(\log n_0c_0/c_1)^{C_3+1}} &\ll_{C_2,C_3} \frac{\tau(r_1)Q_1}{c_0c_1(\log \log X)^{C_3}}
    \end{align*}
    using the usual Taylor series expansion to deal with the $c_i$ inside of the logarithm. Note also that we have extended the range over $n_0$ by positivity of the summand. Now applying Lemma \ref{Siegel--Walfisz1} to the second sum above and summing over $n_0$ we obtain (extending the range as above):
    \[
    \ll_{C_3} \mathop{\sum}_{\substack{(\log X)^{11C_2/16}<n_0\leq Y}}\frac{\tau(r_1)Q_1\|(\log X)^{C_2}/d_1,n_0d_0/d_1\|}{n_0^2c_0^2(\log \|(\log X)^{C_2}/d_1,n_0d_0/d_1\|)^{C_3}}.
    \]
    Here we use the upper bound
    \[
    \|(\log X)^{C_2}/d_1,n_0d_0/d_1\| \leq \frac{n_0c_0}{c_1}
    \]
    for the numerator and the lower bound
    \[
    \|(\log X)^{C_2}/d_1,n_0d_0/d_1\| \geq \frac{n_0d_0}{d_1}
    \]
    and the usual Taylor series expansion for the denominator. Then the above is
    \[
    \mathop{\sum}_{\substack{(\log X)^{11C_2/16}<n_0\leq Y}}\frac{\tau(r_1)Q_1}{n_0c_0c_1(\log n_0)^{C_3}} \ll_{C_2,C_3} \frac{\tau(r_1)Q_1}{c_0c_1(\log \log X)^{C_3}}.
    \]
    This concludes the proof.
\end{proof}

Next we prove a similar result to Lemma \ref{symmetric-hyperbola-sums}: 

\begin{lemma}\label{non-symmetric-hyperbola-sums}
Let $X\geq 3$, $C_1,C_2>0$ and fix some odd integers $Q_0$, $Q_1$ and some $\b{q}\in(\Z/8\Z)^{*4}$. Suppose $\chi_0$ and $\chi_1$ are non-principal characters modulo $Q_0$ and $Q_1$ respectively. Then for any odd integers $1\leq r_0,r_1,r_2,r_3\leq(\log X)^{C_1}$ and any fixed integers $1\leq c_0,c_1,c_2,c_3\leq (\log X)^{C_2/16}$, $1\leq d_0,d_1,d_2,d_3\leq (\log X)^{C_2/4}$ we have
    \[
    \mathop{\sum\sum\sum\sum}_{\substack{\b{n}\in\N^4, \|n_0c_0,n_1c_1\|\cdot\|n_2c_2,n_3c_3\|\leq X\\ \eqref{section3sumconds}}} \frac{\chi_0(n_0)\chi_1(n_1)}{\tau(n_0)\tau(n_1)\tau(n_2)\tau(n_3)} \ll_{C_1,C_2,C_3} \;\frac{\tau(r_0)\tau(r_1)(Q_0+Q_1)X^2}{c_0c_1c_2c_3(\log X)(\log \log X)^{C_3}}
    \]
    for any $C_3>0$, where the implied constant depends at most on $C_1,C_2$ and $C_3$. Here we use \eqref{section3sumconds} with $D=C_2$.
\end{lemma}

\begin{proof}
Once more, let the sum be denoted by $H(X)$. Defining the parameter $Y=\exp((\log X)^{\epsilon})$ for some $0<\epsilon<1$ we use the hyperbola method to write
\[
H(X) = H_0(X) + H_1(X) - H_2(X)
\]
where $H_0(X),H_1(X)$ and $H_2(X)$ are
\[
\mathop{\sum\sum\sum\sum}_{\substack{\|n_0c_0,n_1c_1\|\leq Y\\ \|n_2c_2,n_3c_3\|\leq X/\|c_0n_0,c_1n_1\|\\ \eqref{section3sumconds}}}\frac{\chi_0(n_0)\chi_1(n_1)}{\tau(n_0)\tau(n_1)\tau(n_2)\tau(n_3)},\;\;
\mathop{\sum\sum\sum\sum}_{\substack{\|n_2c_2,n_3c_3\|\leq X/Y\\ \|n_0c_0,n_1c_1\|\leq X/\|n_2c_2,n_3c_3\|\\ \eqref{section3sumconds}}}\frac{\chi_0(n_0)\chi_1(n_1)}{\tau(n_0)\tau(n_1)\tau(n_2)\tau(n_3)}
\]
and
\[
\mathop{\sum\sum\sum\sum}_{\substack{\|n_0c_0,n_1c_1\|\leq Y\\ \|n_2c_2,n_3c_3\|\leq X/Y \\ \eqref{section3sumconds}}}\frac{\chi_0(n_0)\chi_1(n_1)}{\tau(n_0)\tau(n_1)\tau(n_2)\tau(n_3)}
\]
respectively. Following the same strategy as in the proof of Lemma \ref{symmetric-hyperbola-sums}, we may obtain
\begin{align*}
H_1(X),H_2(X)\ll_{C_1,C_2,C_3}\frac{\tau(r_0)\tau(r_1)Q_0Q_1X^2}{c_0c_1c_2c_3(\log X)^{4C_3}}.
\end{align*}
Unlike in Lemma \ref{symmetric-hyperbola-sums} however, $H_0(X)$ and $H_1(X)$ are not symmetric. Trying to use the same method as before for $H_0(X)$ will result in a bound of $X^2(\log X)^2$ which is too big. This is because we lose the information of the characters when we apply the triangle inequality and trivial bounds on the sum over $n_2$ and $n_3$. In order to maintain this information and obtain some saving over the character sum we will instead use Lemma \ref{Siegel--Walfisz1} to provide an asymptotic for the inner sum. Using the trivial bound \eqref{section3smallsum}, we first add in the terms for which $\|n_2d_2,n_3d_3\|\leq (\log X)^{C_2}$ at the cost of a small error term:
\[
H_0(X) = \mathop{\sum\sum\sum\sum}_{\substack{\|n_0c_0,n_1c_1\|\leq Y\\ \|n_2c_2,n_3c_3\|\leq X/\|c_0n_0,c_1n_1\|\\ \|n_0d_0,n_1d_1\|>(\log X)^{C_2}\\ \mathrm{gcd}(n_i,r_i)=1\;\forall 0\leq i\leq 3\\ n_i\equiv q_i\bmod{8}\;\forall 0\leq i\leq 3}}\frac{\chi_0(n_0)\chi_1(n_1)}{\tau(n_0)\tau(n_1)\tau(n_2)\tau(n_3)} + O(X(\log X)^{2C_2}).
\]
The lower bound $\|n_0d_0,n_1d_1\|>(\log X)^{C_2}$ cannot be removed as easily, and is in fact necessary for the result to hold. Using Lemma \ref{Siegel--Walfisz1} on the sum over $n_2$ and $n_3$ we obtain
\begin{align*}
    H_0(X) &= \frac{\mathfrak{S}_0(r_2)\mathfrak{S}_0(r_3) X^2}{\phi(8)^2c_2c_3}\mathop{\sum\sum}_{\substack{n_0c_0,n_1c_1\leq Y\\ \|n_0d_0,n_1d_1\|>(\log X)^{C_2}\\ \mathrm{gcd}(n_i,r_i)=1\;\forall 0\leq i\leq 1\\ n_i\equiv q_i\bmod{8}\;\forall 0\leq i\leq 1}}\frac{\chi_0(n_0)\chi_1(n_1)}{\|n_0c_0,n_1c_1\|^2\tau(n_0)\tau(n_1)(\log (X/c_2c_3\|n_0c_0,n_1c_1\|))} \\
    &+O\left(\frac{X^2(\log \log 3r_2r_3)^{3/2}}{c_2c_3(\log X)^{2}}\mathop{\sum\sum}_{\substack{n_0c_0,n_1c_1\leq Y}}\frac{1}{\|n_0c_0,n_1c_1\|^2\tau(n_0)\tau(n_1)}\right). 
\end{align*}
Here we have to be careful with the logarithmic term in the denominator of the ``main term''. To deal with this we note that, since $c_2c_3\|n_0c_0,n_1c_1\| \ll Y(\log X)^{C_2/8}$, we may write:
\[
\frac{1}{(\log (X/c_2c_3\|n_0c_0,n_1c_1\|))} = \frac{1}{(\log X)}+O\left(\frac{(\log(c_2c_3\|n_0c_0,n_1c_1\|))}{(\log X)^2}\right) = \frac{1}{(\log X)} + O\left(\frac{1}{(\log X)^{2-\epsilon}}\right).
\]
Substituting this into the expression for $H_0(X)$ we will get
\begin{align*}
    H_0(X) \ll& \frac{X^2}{c_2c_3(\log X)}M_0(X) + O\left(\frac{X^2(\log \log 3r_2r_3)^{3/2}}{c_2c_3(\log X)^{2-\epsilon}}E_0(X)\right) 
\end{align*}
where $M_0(X)$ and $E_0(X)$ are
\begin{equation*}
\mathop{\sum\sum}_{\substack{n_0c_0,n_1c_1\leq Y\\ \|n_0d_0,n_1d_1\|>(\log X)^{C_2}\\ \mathrm{gcd}(n_i,r_i)=1\;\forall 0\leq i\leq 1\\ n_i\equiv q_i\bmod{8}\;\forall 0\leq i\leq 1}}\frac{\chi_0(n_0)\chi_1(n_1)}{\|n_0c_0,n_1c_1\|^2\tau(n_0)\tau(n_1)}\;\;\textrm{and}\;\;
\mathop{\sum\sum}_{\substack{n_0c_0,n_1c_1\leq Y\\ \mathrm{gcd}(n_i,r_i)=1\;\forall 0\leq i\leq 1\\ n_i\equiv q_i\bmod{8}\;\forall 0\leq i\leq 1}}\frac{1}{\|n_0c_0,n_1c_1\|^2\tau(n_0)\tau(n_1)}
\end{equation*}
respectively. For $M_0(X)$ we apply Lemma \ref{non-symmetric-hyperbola-sums-mainterm1}, giving:
\[
M_0(X) \ll_{C_2,C_3} \frac{\tau(r_0)\tau(r_1)(Q_0+Q_1)}{c_0c_1(\log \log X)^{C_3}}.
\]
For $E_0(X)$ we can just apply Lemma \ref{maintermlemma1} to see
\[
E_0(X) \ll \left(\frac{\log \log Y}{c_0c_1}\right) \ll_{\epsilon} \left(\frac{\log \log X}{c_0c_1}\right). 
\]
Substituting these bounds into our expression shows that:
\begin{align*}
H_0(X) &\ll_{C_2,C_3}\frac{\tau(r_0)\tau(r_1)(Q_0+Q_1)X^2}{c_0c_1c_2c_3(\log X)(\log \log X)^{C_3}} + \frac{X^2(\log \log X)(\log \log 3r_2r_3)^{3/2}}{c_0c_1c_2c_3(\log X)^{2-\epsilon}}\\
&\ll_{C_2,C_3}\frac{\tau(r_0)\tau(r_1)(Q_0+Q_1)X^2}{c_0c_1c_2c_3(\log X)(\log \log X)^{C_3}}
\end{align*}
concluding the proof.
\end{proof}

Lemma \ref{non-symmetric-hyperbola-sums} is only effective when the conductors, $Q_i$, are bounded by a power of $\log\log X$. Therefore, we must turn to other methods to deal with the larger parts of such averages. This last ingredient is precisely Corollary \ref{AveragingovermediumconductorsIntro}. We are now ready to prove the main result of this section.

\begin{proposition}\label{Asymmetrictypeaverage1}
    Let $X\geq 3$, $C_1,C_2>0$ be such that $(C_1\log\log X)^{C_2}>2$. Fix some odd square-free integers $Q_1,Q_2,Q_3\in\N$ such that $Q_1\leq (\log\log X)^{C_2}$, and some $\b{q}\in(\Z/8\Z)^{*4}$, $\Tilde{\b{q}}\in(\Z/8\Z)^{*2}$. Suppose $\chi_2$ and $\chi_3$ are characters modulo $Q_2$ and $Q_3$ respectively. Fixing any odd integers $1\leq r_0,r_1,r_2,r_3\leq (\log X)^{C_1}$ such that $\mathrm{gcd}(Q_1,r_0r_1r_2r_3)=\mathrm{gcd}(Q_2Q_3,r_2r_3)=1$ and fixing any $1\leq c_0,c_1,c_2,c_3\leq (\log X)^{C_2/32}$, $1\leq d_0,d_1,d_2,d_3\leq (\log X)^{C_2/4}$ we define, for any $\b{m}\in\N^2$
    \[
    H'(X,\b{m}) = \mathop{\sum\sum\sum\sum}_{\substack{\b{n}\in\N^4, \|n_0d_0,n_1d_1\|,\|n_2d_2,n_3d_3\|>(\log X)^{C_2}\\ \|n_0c_0,n_1c_1\|\cdot\|n_2m_2c_2,n_3m_3c_3\|\leq X\\ \mathrm{gcd}(n_i,2r_i)=1\;\forall\;0\leq i\leq 3\\ \b{n}\equiv \b{q}\bmod{8}}} \frac{\psi_{m_2m_3}(n_2n_3)}{\tau(n_0)\tau(n_1)\tau(n_2)\tau(n_3)}.
    \]
    Then,
    \begin{align*}
    \mathop{\sum\sum}_{\substack{\b{m}\in\N^2, \|m_2,m_3\|\leq (\log X)^{C_2}\\ \mathrm{gcd}(m_i,2Q_1Q_2Q_3r_i)=1\;\forall 2\leq i\leq 3\\ \b{m}\equiv \Tilde{\b{q}}\bmod{8}\\ Q_1m_2m_3\neq 1}}\frac{\mu^2(m_2m_3)\chi_2(m_2)\chi_3(m_3)}{\tau(m_2)\tau(m_3)}H'(X,\b{m}) \ll_{C_2} \frac{\tau(r_0)\tau(r_1)X^2}{c_0c_1c_2c_3(\log X)(\log \log X)^{C_3}}.
    \end{align*}
    where $C_3 = C_2/2-1$ and where the implied constant depends at most on $C_1$ and $C_2$.
\end{proposition}

\begin{proof}
Denote $W=(\log X)^{C_2}$ and recall that we have defined $\psi_{m}(\cdot)$ as $\left(\frac{\cdot}{m}\right)$ or $\left(\frac{m}{\cdot}\right)$ generically. We will assume the former, and we remark that quadratic reciprocity may be used to swap between the two cases. Call the sum under consideration $S(X)$. Then we write
\[
S(X) = S_1(X) + S_2(X)
\]
where $S_1(X)$ and $S_2(X)$ are defined as $S(X)$ with the extra conditions $\|m_2,m_3\|\leq (\log W)^{C_2}$ and $(\log W)^{C_2}<\|m_2,m_3\|\leq W$ respectively. First we deal with $S_1(X)$. For each fixed $\b{m}$ such that $\mathrm{gcd}(m_i,2Q_1Q_2Q_3r_i)=1$, $\mu^2(m_2m_3)=1$ and $Q_1m_2m_3\neq 1$, the sum $H'(X,\b{m})$ is precisely of the form considered in Lemma \ref{non-symmetric-hyperbola-sums}. Furthermore, the range of $m_2$ and $m_3$ in $S_1(X)$ is small enough for this lemma to be effective. Thus,
\begin{align*}
S_1(X) &\ll_{C_2,C_3} \mathop{\sum\sum}_{\substack{\b{m}\in\N^2, \|m_2,m_3\|\leq (\log W)^{C_2}}} \frac{\mu^2(m_2m_3)\tau(r_0)\tau(r_1)Q_1m_2m_3X^2}{m_2m_3\tau(m_2)\tau(m_3)c_0c_1c_2c_3(\log X)(\log W)^{4C_2}}\\
&\ll_{C_2,C_3} \frac{\tau(r_0)\tau(r_1)X^2}{c_0c_1c_2c_3(\log X)(\log W)^{C_3}}.
\end{align*}
Next we deal with $S_2(X)$. We first use the hyperbola method on the $H'(X,\b{m})$ terms again with the parameter $Y=\exp((\log X)^{\epsilon})$ to obtain 
\[
H'(X,\b{m}) = H_0'(X,\b{m}) + H_1'(X,\b{m}) - H_2'(X,\b{m})
\]
where
\[
H_0'(X,\b{m}) = \mathop{\sum\sum\sum\sum}_{\substack{\|n_0c_0,n_1c_1\|\leq Y\\ \|n_2m_2c_2,n_3m_3c_3\|\leq X/\|n_0c_0,n_1c_1\| \\  \eqref{section3sumconds}}}\frac{1}{\tau(n_0)\tau(n_1)\tau(n_2)\tau(n_3)}\left(\frac{n_0n_1}{Q_1m_2m_3}\right),
\]
and $H_1'(X,\b{m}),H_2'(X,\b{m})$ are defined as $H_0'(X,\b{m})$ with the height conditions $\{\|n_0c_0,n_1c_1\|\leq Y,\; \|n_2m_2c_2,n_3m_3c_3\|\leq X/\|n_0c_0,n_1c_1\|\}$ replaced with
\[
\{\|n_2m_2c_2,n_3m_2c_3\|\leq X/Y,\; \|n_0c_0,n_1c_1\|\leq X/\|n_2m_2c_2,n_3m_2c_3\|\}
\]
and
\[
\{\|n_0c_0,n_1c_1\|\leq Y,\; \|n_2m_2c_2,n_3m_2c_3\|\leq X/Y\}
\]
respectively. $H_1'(X,\b{m})$ and $H_2'(X,\b{m})$ may be dealt with using Lemma \ref{Siegel--Walfisz1} since the ranges of the sums over $n_0$ and $n_1$ are guaranteed to be exponential in the size of $Q_1m_2m_3$. The conditions on $Q_1,m_2$ and $m_3$ guarantee that $\psi_{Q_1m_2m_3}$ is non-principal. Then Lemma \ref{Siegel--Walfisz1} will give arbitrary logarithmic saving in the sums over $n_0$ and $n_1$ so that, upon summing over $n_2$ and $n_3$ we obtain
\[
H_1'(X,\b{m}),H_2'(X,\b{m})\ll_{C_1,C_2} \frac{\tau(r_0)\tau(r_1)Q_1^2m_2^2m_3^2X^2}{m_2m_3c_0c_1c_2c_3(\log X)^{6C_2}} \ll_{C_1,C_2} \frac{Q_1^2m_2m_3X^2}{c_0c_1c_2c_3(\log X)^{5C_2}}.
\]
See the bounds for \eqref{non-principalsymmetrichyperbolicregionsum0} and \eqref{typicalrectangularregioncharactersum} in Lemma \ref{symmetric-hyperbola-sums} for analogous proofs. Summing these trivially over the $m_i$ will then give
\[
\mathop{\sum\sum}_{\substack{\b{m}\in\N^2\\ (\log W)^{C_3}<\|m_2,m_3\|\leq W\\ \mathrm{gcd}(m_i,2Q_1Q_2Q_3r_i)=1\;\forall 2\leq i\leq 3\\ \b{m}\equiv \Tilde{\b{q}}\bmod{8}}}\frac{\mu^2(m_2m_3)\chi_2(m_2)\chi_3(m_3)}{\tau(m_2)\tau(m_3)}H_1'(X,\b{m}) \ll_{C_2,C_3} \frac{\tau(r_0)\tau(r_1)X^2}{c_0c_1c_2c_3(\log X)^{C_2}}
\]
and likewise,
\begin{align*}
\mathop{\sum\sum}_{\substack{\b{m}\in\N^2\\ (\log W)^{C_3}<\|m_2,m_3\|\leq W \\ \mathrm{gcd}(m_i,2Q_1Q_2Q_3r_i)=1\;\forall 2\leq i\leq 3\\ \b{m}\equiv \Tilde{\b{q}}\bmod{8}}}\frac{\mu^2(m_2m_3)\chi_2(m_2)\chi_3(m_3)}{\tau(m_2)\tau(m_3)}H_2'(X,\b{m}) \ll_{C_2,C_3} \frac{\tau(r_0)\tau(r_1)X^2}{c_0c_1c_2c_3(\log X)^{C_2}}
\end{align*}
since $W=(\log X)^{C_2}$. We now turn to $H_0'(X,\b{m})$. We add in the terms for which $\|n_2d_2,n_3d_3\|\leq(\log X)^{C_2}$ at the cost of a small error term (see, for example \eqref{section3smallsum}):
\[
H_0'(X,\b{m}) = \mathop{\sum\sum\sum\sum}_{\substack{\|n_0c_0,n_1c_1\|\leq Y\\ \|n_2m_2c_2,n_3m_3c_3\|\leq X/\|n_0c_0,n_1c_1\| \\  \|n_0d_0,n_1d_1\|>W \\ \mathrm{gcd}(n_i,r_i)=1\;\forall 0\leq i\leq 3\\ n_i\equiv q_i\bmod{8}\;\forall 0\leq i\leq 3}}\frac{1}{\tau(n_0)\tau(n_1)\tau(n_2)\tau(n_3)}\left(\frac{n_0n_1}{Q_1m_2m_3}\right) + O(X(\log X)^{2C_2}).
\]
Next we apply Lemma \ref{Siegel--Walfisz1} for non-principal characters with $Q=1$ and $C=3/2$ on the sum over $n_2$ and $n_3$, allowing us to preserve the Jacobi symbol. Upon using the standard Taylor series method to the logarithmic factors (as in Lemma \ref{non-symmetric-hyperbola-sums} since $m_2m_3c_2c_3\|n_0c_0,n_1c_1\|\leq Y(\log X)^{3C_1}=o(X)$), we obtain
\[
H_0'(X,\b{m}) = H_{00}'(X,\b{m}) + O(H_{01}'(X,\b{m}))
\]
where
\[
H_{00}'(X,\b{m}) = \frac{\mathfrak{S}_0(r_2)\mathfrak{S}_0(r_3) X^2}{\phi(8)^2 m_2m_3c_2c_3(\log X)}\left(\mathop{\sum\sum}_{\substack{\|n_0c_0,n_1c_1\|\leq Y\\ \|n_0d_0,n_1d_1\|>W\\ \mathrm{gcd}(n_i,r_i)=1\;\forall 0\leq i\leq1 \\ n_i\equiv q_i\bmod{8}\;\forall 0\leq i\leq 1}}\frac{1}{\|n_0c_0,n_1c_1\|^2\tau(n_0)\tau(n_1)}\left(\frac{n_0n_1}{Q_1m_2m_3}\right)\right),
\]
and
\[
H_{01}'(X,\b{m}) = \frac{X^2}{m_2m_3c_2c_3(\log X)^{2-\epsilon}}\mathop{\sum\sum}_{\substack{\|n_0c_0,n_1c_1\|\leq Y\\ \|n_0d_0,n_1d_1\|>W}}\frac{1}{\|n_0c_0,n_1c_1\|^2\tau(n_0)\tau(n_1)}.
\]
Note that we have used $r_i\ll (\log X)^{C_1}$ to absorb $(\log\log 3r_2r_3)^{3/2}$ and $\tau(r_2)\tau(r_3)$ into $(\log X)^{\epsilon}$. Let us fist deal with the $H_{01}'(X,\b{m})$. By Lemma \ref{maintermlemma1} the sum is $O\left(\frac{\log \log X}{c_0c_1}\right)$, so that overall,
\[
H_{01}'(X,\b{m}) \ll \frac{X^2(\log\log X)}{c_0c_1c_2c_3m_2m_3(\log X)^{2-\epsilon}}.
\]
Summing over $\b{m}$ will give
\[
\ll \mathop{\sum\sum}_{\substack{\|m_2,m_3\|\leq W}} \lvert H_{01}'(X,\b{m})\rvert \ll_{C_2,C_3} \frac{X^2(\log\log X)^4}{c_0c_1c_2c_3(\log X)^{2-\epsilon}}.
\]
To deal with $H_{00}'(X,\b{m})$ we use the averaging over $\b{m}$. Specifically, we are left to bound
\[
S_{200}(X) = \mathop{\sum\sum}_{\substack{\b{m}\in\N^2\\ (\log W)^{C_2}<\|m_2,m_3\|\leq W \\ \mathrm{gcd}(m_i,2Q_1Q_2Q_3r_i)=1\;\forall 2\leq i\leq 3\\ \b{m}\equiv \Tilde{\b{q}}\bmod{8}}}\frac{\mu^2(m_2m_3)\chi_2(m_2)\chi_3(m_3)}{\tau(m_2)\tau(m_3)}H_{00}'(X,\b{m}).
\]
This is bounded by
\begin{align*}
    \frac{X^2}{c_2c_3(\log X)}\left\lvert\mathop{\sum\sum}_{\substack{\b{m}\in\N^2,\b{n}\in\N^{4}\\ (\log W)^{C_2}<\|m_2,m_3\|\leq W \\ \mathrm{gcd}(m_i,2Q_1Q_2Q_3r_i)=1\;\forall 2\leq i\leq 3\\ \b{m}\equiv \Tilde{\b{q}}\bmod{8}}}\mathop{\sum\sum}_{\substack{\|n_0c_0,n_1c_1\|\leq Y\\ \|n_0d_0,n_1d_1\|>W \\ \mathrm{gcd}(n_i,r_i)=1\;0\leq i\leq 1\\ n_i\equiv q_i\bmod{8}\;0\leq i\leq 1}}\hspace{-4pt}\frac{\mu^2(m_2m_3)\chi_2(m_2)\chi_3(m_3)\left(\hspace{-2pt}\frac{n_0n_1}{Q_1m_2m_3}\hspace{-2pt}\right)}{m_2m_3\|n_0c_0,n_1c_1\|^2\tau(m_2m_3)\tau(n_0)\tau(n_1)}\right\rvert
\end{align*}
To begin we will write $m=m_2m_3$ and define
\[
\bar{\tau}(m) = \mathop{\sum\sum}_{\substack{m_2m_3=m, \|m_2,m_3\|>(\log W)^{C_2}\\ \mathrm{gcd}(m_i,2Q_1Q_2Q_3r_i)=1\;\forall 2\leq i\leq 3\\ m_i\equiv \Tilde{q}_i\bmod{8}\;\forall 2\leq i\leq 3}}\chi_2(m_2)\chi_3(m_3).
\]
Then, by rewriting, we see that $S_{200}(X)$ is
\begin{align*}
    \ll \frac{X^2}{c_2c_3(\log X)}\left\lvert\sum_{\substack{(\log W)^{C_2}< m \leq W^2}}\mathop{\sum\sum}_{\substack{\|n_0c_0,n_1c_1\|\leq Y\\ \|n_0d_0,n_1d_1\|>W \\ \mathrm{gcd}(n_i,r_i)=1\;\forall 0\leq i\leq 1\\ n_i\equiv q_i\bmod{8}\;\forall 0\leq i\leq 1}}\frac{\mu^2(m)\bar{\tau}(m)}{m\|n_0c_0,n_1c_1\|^2\tau(m)\tau(n_0)\tau(n_1)}\left(\frac{n_0n_1}{Q_1m}\right)\right\rvert.
\end{align*}
Next we split the sum over $n_0$ and $n_1$ into a region where $n_1c_1\leq n_0c_0$ and a second region where $n_0c_0<n_1c_1$. The sum over each region will be of the same order, so that $S_{200}(X)$ becomes
\begin{align*}
    \ll \frac{X^2}{c_2c_3(\log X)}\left\lvert\sum_{\substack{W^{3/4}/c_0<n_0\leq Y/c_0\\ \mathrm{gcd}(n_0,r_0)=1\\n_0\equiv q_0\bmod{8}}}\frac{1}{n_0^2c_0^2\tau(n_0)}\sum_{\substack{(\log W)^{C_2}< m \leq W^2}}\sum_{\substack{n_1c_1\leq n_0c_0\\ \|n_0d_0,n_1d_1\|>W \\ \mathrm{gcd}(n_1,r_1)=1\\ n_1\equiv q_1\bmod{8}}}\frac{\mu^2(m)\bar{\tau}(m)}{m\tau(m)\tau(n_1)}\left(\frac{n_0n_1}{Q_1m}\right)\right\rvert.
\end{align*}
Here we note that, although $n_0c_0\geq n_1c_1$, we may still have $n_0d_0<n_1d_1$; however, this can only occur when
\[
\frac{c_1}{c_0}\leq \frac{n_0}{n_1} < \frac{d_0}{d_1}.
\]
If this were the case, then we use the fact that $\|n_0d_0,n_1d_1\|>W$ to assert
\[
\frac{Wc_1}{c_0} < \frac{n_1d_1c_1}{c_0} \leq n_0d_1,
\]
from which it follows that
\[
n_0c_0 \geq \frac{Wc_1}{d_1} \geq W^{3/4}
\]
giving the lower bound in the sum over $n_0$. The bound $\|n_0d_0,n_1d_1\|>W$ is maintained as a condition on $n_1$. Next we define
\[
a_m = \frac{\mu^2(m)\bar{\tau}(m)}{\tau(m)}\left(\frac{n_0}{m}\right),
\]
and
\[
b_{n_1} = \mathds{1}(\mathrm{gcd}(n_1,r_1)=1)\mathds{1}(n_1\equiv q_1\bmod{8})\mathds{1}(\|n_0d_0,n_1d_1\|>W)\left(\frac{n_1}{Q_1}\right).
\]
Then, using the triangle inequality,
\begin{align*}
    S_{200}(X) \ll \frac{X^2}{c_2c_3(\log X)}\sum_{\substack{W^{3/4}/c_0<n_0\leq Y/c_0}}\frac{1}{n_0^2c_0^2\tau(n_0)}\left\lvert\mathop{\sum\sum}_{\substack{(\log W)^{C_2}< m \leq W^2\\ n_1\leq n_0c_0/c_1}}\frac{a_m b_{n_1}}{m\tau(n_1)}\left(\frac{n_1}{m}\right)\right\rvert.
\end{align*}
Finally, by applying Corollary \ref{AveragingovermediumconductorsIntro}, we obtain
\[
S_{200}(X)\ll_{C_2,C_3} \frac{X^2}{c_0c_1c_2c_3(\log X)(\log \log X)^{C_2/2-1}},
\]
which is sufficient.
\end{proof}

\section{Character sums over hyperbolic regions II}\label{hyperbolicanalysis2}
In this section we deal sums where characters are arranged differently with respect to the hyperbolic height conditions. The type of sum considered is of the form
\begin{equation}\label{Generalsumtype2}
\mathop{\sum\sum\sum\sum}_{\substack{\|n_0n_1,n_2n_3\| \leq X}}\frac{\chi(n_0n_2)\psi(n_1n_3)}{\tau(n_0)\tau(n_1)\tau(n_2)\tau(n_3)}
\end{equation}
where $\chi$ and $\psi$ are some Dirichlet characters. Just as in \S \ref{hyperbolicanalysis1} we have three cases:
\begin{itemize}
    \item[(a)] Main Term: both $\chi$ and $\psi$ are principal;
    \item[(b)] Small Conductor -- Symmetric Hyperbola Method: both $\chi$ and $\psi$ are non-principal;
    \item[(c)] Small Conductor -- Non-Symmetric Hyperbola Method: only one of $\chi$ or $\psi$ are non-principal
\end{itemize}
Recall that the main term in this case may be seen to be of order $X^2$. As mentioned in the introduction, we will ignore such cases. We set up the preliminaries for this in the first subsection and handle the symmetric and non-symmetric cases using the results of \S\ref{technicalstuff}.

\subsection{Sums Over Fixed Conductors}
When dealing with the counting problem \ref{thecountingproblem}, we will have to bound contributions from the character sums where the only characters present are characters of modulus $8$. Such contributions require their own result. Note that the method used in this section may be seen to be a simpler version of the method used in \S\ref{Lfunctionmethodsection}, where we will need to average over the modulus of the characters.
\begin{lemma}\label{fixedconductorlemmaforvanishingmainterms}
    Let $X\geq 3$, $C_1,C_2>0$. Suppose $\chi_{0},\chi_{1},\chi_{2}$ and $\chi_{3}$ are Dirichlet characters modulo $8$ such that $\chi_i$ and $\chi_j$ are non-principal for some pair $(i,j)\in\{0,1\}\times\{2,3\}$. Then for any odd integers $1\leq r_0,r_1,r_2,r_3\leq (\log X)^{C_1}$ and any integers $1\leq c_{01},c_{23},M \leq (\log X)^{C_2}$ we have,
    \begin{equation*} \mathop{\sum\sum\sum\sum}_{\substack{\|n_0n_1c_{01},n_2n_3c_{23}\|\cdot M \leq X \\ \mathrm{gcd}(n_i,2r_i)=1\forall\;0\leq i\leq 3}}\frac{\chi_{0}(n_0)\chi_{2}(n_2)\chi_{1}(n_1)\chi_{3}(n_3)}{\tau(n_0)\tau(n_1)\tau(n_2)\tau(n_3)} \ll_{C_1,C_2} \frac{\tau(r_0)\tau(r_1)\tau(r_2)\tau(r_3)X^2}{c_{01}c_{23}M^2(\log X)}.
    \end{equation*}
    where the implied constant depends at most on $C_1$ and $C_2$.
\end{lemma}

\begin{proof}
    We write the sum under consideration as $H_{01}(X)H_{23}(X)$
    where,
    \[
    H_{01}(X) = \mathop{\sum\sum}_{\substack{n_0n_1\leq X/c_{01}M,\\ \mathrm{gcd}(n_i,2r_i)=1\forall\;i\in\{0,1\}}}\frac{\chi_{0}(n_0)\chi_{1}(n_1)}{\tau(n_0)\tau(n_1)}\;\;\text{and}\;\;
    H_{23}(X) = \mathop{\sum\sum}_{\substack{n_2n_3\leq X/c_{23}M \\ \mathrm{gcd}(n_i,2r_i)=1\forall\;i\in\{2,3\}}}\frac{\chi_{2}(n_2)\chi_{3}(n_3)}{\tau(n_2)\tau(n_3)}.
    \]
    These are symmetric and thus we will focus on $H_{01}(X)$ and note that any bound for this may also be obtained for $H_{23}(X)$. Further, we will assume without loss in generality that $\chi_{0}$ and $\chi_{2}$ are non-principal. Letting $Y=\exp((\log X)^{1/6})$ we use the classical hyperbola method to write
    \begin{align*}
    H_{01}(X) =& \sum_{\substack{n_0\leq Y\\ \mathrm{gcd}(n_0,2r_0)=1}}\hspace{-4pt}\frac{\chi_{0}(n_0)}{\tau(n_0)}\hspace{-4pt}\sum_{\substack{n_1\leq X/n_0c_{01}M\\ \mathrm{gcd}(n_1,2r_1)=1}}\hspace{-4pt}\frac{\chi_{1}(n_1)}{\tau(n_1)} + \hspace{-3pt}\sum_{\substack{n_1\leq X/c_{01}MY\\ \mathrm{gcd}(n_1,2r_1)=1}}\hspace{-4pt}\frac{\chi_{1}(n_1)}{\tau(n_1)}\hspace{-4pt}\sum_{\substack{n_0\leq X/n_1c_{01}M\\ \mathrm{gcd}(n_0,2r_0)=1}}\hspace{-4pt}\frac{\chi_{0}(n_0)}{\tau(n_0)}\\ &- \mathop{\sum\sum}_{\substack{n_0\leq Y,n_1\leq X/c_{01}MY\\ \mathrm{gcd}(n_i,2r_i)=1}}\frac{\chi_{0}(n_0)}{\tau(n_0)}\frac{\chi_{1}(n_1)}{\tau(n_1)}.
    \end{align*}
    For the second and third sums we use Lemma \ref{Siegel--Walfisz1} for the sum over $\chi_{0}(n_0)$ with $Q=8$ and $C=2025$ in this lemma, noting that for the second sum we use,
    \[
    \frac{1}{(\log X/n_1c_{01}M)^{2025}}\ll \frac{1}{(\log Y)^{2025}}\ll \frac{1}{(\log X)^{2025/6}}
    \]
    since $n_1\leq X/Yc_{01}M$. In each case, upon summing over $n_1$, we obtain
    \[
    \ll \frac{\tau(r_0)X}{c_{01}M(\log X)^{2025/6}}.
    \]
    We are left with
    \[
    H_{01}(X) = \sum_{\substack{n_0\leq Y\\ \mathrm{gcd}(n_0,2r_0)=1}}\hspace{-4pt}\frac{\chi_{0}(n_0)}{\tau(n_0)}\hspace{-4pt}\sum_{\substack{n_1\leq X/n_0c_{01}M\\ \mathrm{gcd}(n_1,2r_1)=1}}\hspace{-4pt}\frac{\chi_{1}(n_1)}{\tau(n_1)} + O\left(\frac{\tau(r_0)X}{c_{01}M(\log X)^{2025/6}}\right).
    \]
    This error term is sufficient since $\tau(r_0)\ll_{C_1} (\log X)^{1/6}$. Now, if $\chi_{1}$ is non-principal then this remaining sum may be handled in the same way as the second. We therefore assume that is the principal character modulo $8$. Then, given the height conditions, we may use Lemma \ref{Siegel--Walfisz1} for principal characters modulo $8$ on the inner sum over $n_1$ with $C=2025$ sufficiently large. This will give $H_{01}(X)$ equal to 
    \begin{align*}
    \frac{\mathfrak{S}_0(r_1)X}{c_{01}M}\sum_{\substack{n_0\leq Y\\ \mathrm{gcd}(n_0,2r_0)=1}}\hspace{-4pt}\frac{\chi_{0}(n_0)}{n_0\tau(n_0)(\log X/n_0c_{01}M)^{1/2}}+ O\left(\frac{X(\log\log 3r_1)^{3/2}}{c_{01}M(\log X)^{3/2}}\sum_{\substack{n_0\leq Y\\ \mathrm{gcd}(n_0,2r_0)=1}}\hspace{-4pt}\frac{1}{n_0\tau(n_0)}\right).
    \end{align*}
    Using the bound
    \[
    \sum_{\substack{n_0\leq Y\\ \mathrm{gcd}(n_0,2r_0)=1}}\hspace{-4pt}\frac{1}{n_0\tau(n_0)} \ll (\log Y)^{1/2}\ll (\log X)^{1/2}
    \]
    and the typical Taylor series expansion
    \[
    \frac{1}{(\log X/n_0c_{01}M)^{1/2}} = \frac{1}{(\log X)^{1/2}}+O\left(\frac{1}{(\log X)^{4/3}}\right)
    \]
    (the latter a result of $n_0c_{01}M\leq Y(\log X)^{2C_2}$), this becomes
    \[
    H_{01}(X) \ll \frac{X}{c_{01}M(\log X)^{1/2}}\left\lvert\sum_{\substack{n_0\leq Y\\ \mathrm{gcd}(n_0,2r_0)=1}}\hspace{-4pt}\frac{\chi_{0}(n_0)}{n_0\tau(n_0)}\right\rvert + O\left(\frac{X(\log\log 3r_1)^{3/2}}{c_{01}M(\log X)^{5/4}}\right).
    \]
    Now let us consider the remaining sum over $n_0$. To do this first consider the Dirichlet series
    \[
    D(s,\chi_{0}) = \sum_{\substack{n=1\\ \mathrm{gcd}(n,2r)=1}}^{\infty}\hspace{-4pt}\frac{\chi_{0}(n)}{n^s\tau(n)}.
    \]
    Using the Euler product we may write
    \[
    D(s,\chi_{0}) = P(s,r,\chi_{0})R(s,\chi_{0})L(s,\chi_{0})^{1/2}
    \]
    where $P(s,r,\chi_0)$ and $R(s,\chi_0)$ are 
    \[ \prod_{\substack{p\;\text{prime}\\p|r}}\left(1+\sum_{j=1}^{\infty}\frac{\chi_{0}(p)^j}{(j+1)p^{js}}\right)^{-1}\;\;\textrm{and}\;\;\prod_{p\;\text{prime}}\left(1+\sum_{j=1}^{\infty}\frac{\chi_{0}(p)^j}{(j+1)p^{js}}\right)\left(1-\frac{\chi_{0}(p)}{p^s}\right)^{1/2}
    \]
    respectively, the second product converging absolutely when $\Re(s)>1/2$, and
    \[
    L(s,\chi_{0}) = \prod_{p\;\text{prime}}\left(1-\frac{\chi_{0}(p)}{p^s}\right)^{-1}
    \]
    is the $L$-function for the character $\chi_{0}$. It follows from this decomposition that the Dirichlet series converges whenever $\Re(s)>1/2$ and $L(s,\chi_{0})\neq 0$. Using the zero free region for $L$-functions of primitive characters and Siegel's Theorem it follows, in particular, that $D(1,\chi_{0})$ converges and that
    \[
    D(1,r,\chi_{0}) = P(1,r,\chi_{0})R(1,\chi_{0})L(1,\chi_{0})^{1/2} \ll \tau(r)
    \]
    Therefore,
    \[
    \sum_{\substack{n_0\leq Y\\ \mathrm{gcd}(n_0,2r_0)=1}}\hspace{-4pt}\frac{\chi_{0}(n_0)}{n_0\tau(n_0)} \ll \tau(r_0)
    \]
    from which it follows that
    \[
    H_{01}(X) \ll \frac{\tau(r_0)\tau(r_1)X}{c_{01}M(\log X)^{1/2}}.
    \]
    Similarly,
    \[
    H_{23}(X) \ll \frac{\tau(r_2)\tau(r_3)X}{c_{23}M(\log X)^{1/2}}.
    \]
\end{proof}

\subsection{Small Conductor -- Symmetric Hyperbola Method}
As with case $(b)$ of \S \ref{hyperbolicanalysis1}, we only need to apply Lemmas \ref{Siegel--Walfisz1} and \ref{hyperbolamethod} appropriately.

\begin{lemma}\label{other-symmetric-type-sums}
    Let $X\geq 3$, $C_1,C_2>0$, $Q_{02},Q_{13}$ be odd integers and $\b{q}\in(\Z/8\Z)^{*4}$. Suppose $\chi_{02}$, $\chi_{13}$ are non-principal Dirichlet characters modulo $Q_{02}$, $Q_{13}$ respectively. Then for any odd integers $1\leq r_0,r_1,r_2,r_3\leq (\log X)^{C_1}$ such that $\mathrm{gcd}(Q_{ij},2r_ir_j)=1$ whenever $(i,j)\in\{(0,2),(1,3)\}$ and any integers $1\leq c_{01},c_{23},M \leq (\log X)^{C_2}$ we have,
    \begin{equation*} \mathop{\sum\sum\sum\sum}_{\substack{\|n_0n_1c_{01},n_2n_3c_{23}\|\cdot M \leq X \\ \mathrm{gcd}(n_i,2r_i)=1\forall\;0\leq i\leq 3\\ n_i\equiv q_i\bmod{8}\forall\;0\leq i\leq 3}}\frac{\chi_{02}(n_0n_2)\chi_{13}(n_1n_3)}{\tau(n_0)\tau(n_1)\tau(n_2)\tau(n_3)} \ll_{C_2,C_3} \frac{Q_{02}Q_{13}X^2}{c_{01}c_{23}M^2(\log X)^{C_3}}.
    \end{equation*}
    for any $C_3>0$, where the implied constant depends at most on the $C_i$.
\end{lemma}

\begin{proof}
    Write the sum under consideration as $H_{01}(X)H_{23}(X)$
    where,
    \[
    H_{01}(X) = \mathop{\sum\sum}_{\substack{n_0n_1\leq X/c_{01}M,\\ \mathrm{gcd}(n_i,2r_i)=1\forall\;i\in\{0,1\}\\ n_i\equiv q_i\bmod{8}\forall\;i\in\{0,1\}}}\frac{\chi_{02}(n_0)\chi_{13}(n_1)}{\tau(n_0)\tau(n_1)}\;\;\textrm{and}\;\;
    H_{23}(X) = \mathop{\sum\sum}_{\substack{n_2n_3\leq X/c_{23}M \\ \mathrm{gcd}(n_i,2r_i)=1\forall\;i\in\{2,3\}\\ n_i\equiv q_i\bmod{8}\forall\;i\in\{2,3\}}}\frac{\chi_{02}(n_2)\chi_{13}(n_3)}{\tau(n_2)\tau(n_3)}.
    \]
    These sums are symmetric and so we focus on $H_{01}(X)$. The hyperbola method gives
    \begin{align*}
    H_{01}(X) =& \sum_{\substack{n_0\leq X^{1/2}/c_{01}^{1/2}M^{1/2}\\ \mathrm{gcd}(n_0,2r_0)=1\\ n_0\equiv q_0\bmod{8}}}\hspace{-4pt}\frac{\chi_{02}(n_0)}{\tau(n_0)}\hspace{-4pt}\sum_{\substack{n_1\leq X/n_0c_{01}M\\ \mathrm{gcd}(n_1,2r_1)=1\\ n_1\equiv q_1\bmod{8}}}\hspace{-4pt}\frac{\chi_{13}(n_1)}{\tau(n_1)} + \hspace{-3pt}\sum_{\substack{n_1\leq X^{1/2}/c_{01}^{1/2}M^{1/2}\\ \mathrm{gcd}(n_1,2r_1)=1\\ n_1\equiv q_1\bmod{8}}}\hspace{-4pt}\frac{\chi_{13}(n_1)}{\tau(n_1)}\hspace{-4pt}\sum_{\substack{n_0\leq X/n_1c_{01}M\\ \mathrm{gcd}(n_0,2r_0)=1\\ n_0\equiv q_0\bmod{8}}}\hspace{-4pt}\frac{\chi_{02}(n_0)}{\tau(n_0)}\\ &- \mathop{\sum\sum}_{\substack{n_0,n_1\leq X^{1/2}/c_{01}^{1/2}M^{1/2}\\ \mathrm{gcd}(n_i,2r_i)=1\\ n_i\equiv q_i\bmod{8}}}\frac{\chi_{02}(n_0)}{\tau(n_0)}\frac{\chi_{13}(n_1)}{\tau(n_1)}.
    \end{align*}
    The last of these sums may be written as the product of the sum over $n_0$ and the sum over $n_1$. Since both $\chi_{02}$ and $\chi_{13}$ are non-principal, we use Lemma \ref{Siegel--Walfisz1} on each part of this product and multiply the results to show that the contribution from this sum is
    \[
    \ll_{C_1,C_2} \frac{\tau(r_0)\tau(r_1)Q_{02}Q_{13}X}{c_{01}M(\log X/c_{01}M)^{2C_3+1}} \ll_{C_2,C_3} \frac{Q_{02}Q_{13}X}{c_{01}M(\log X)^{C_3+1}}.
    \]
    This last bound is obtained using the assumption $c_{01},M\leq (\log X)^{C_2}$. The first two sums in the expression for $H_{01}$ are dealt with in the same way since both characters are non-principal. Looking at the first sum, we use Lemma \ref{Siegel--Walfisz1} for the sum over $n_1$. This leads to
    \[
    \ll_{C_3} \sum_{\substack{n_0\leq X^{1/2}/c_{01}^{1/2}M^{1/2}}} \frac{\tau(r_1)Q_{13}X}{n_0c_{01}M(\log X/n_0c_{01}M)^{2C_3+2}}.
    \]
    Now, since $n_0\leq X^{1/2}/c_{01}^{1/2}M^{1/2}$ and $c_{01},M\leq (\log X)^{C_2}$ it follows that the first sum is then bounded by
    \begin{align*}
    \ll_{C_2,C_3} \sum_{\substack{n_0\leq X^{1/2}}} \frac{\tau(r_1)Q_{13}X}{n_0c_{01}M(\log X)^{2C_3+2}} \ll_{C_2,C_3} \frac{Q_{13}X}{c_{01}M(\log X)^{C_3+1}}
    \end{align*}
    Thus
    \[
    H_{01}(X) \ll_{C_2,C_3} \frac{Q_{02}Q_{13}X}{c_{01}M(\log X)^{C_3+1}}.
    \]
    Putting this together with the trivial bound $\frac{X(\log X)}{c_{23}M}$ for $H_{23}(X)$ gives the result.
\end{proof}

We conclude this subsection with a final averaging result. Its proof is a direct application of Lemma \ref{other-symmetric-type-sums}.

\begin{proposition}\label{symmetrictypeaverage2}
    Let $X\geq 3$, $C_1,C_2,C_3>0$, let $Q_{02},Q_{13}$ be odd integers and take $\b{q}\in(\Z/8\Z)^{*4}$, $\Tilde{\b{q}}\in(\Z/8\Z)^{*2}$. Let $1\leq r_0,r_1,r_2,r_3\leq (\log X)^{C_1}$ be odd integers such that $\mathrm{gcd}(Q_{ij},2r_ir_j)=1$ for $(i,j)\in\{(0,2),(1,3)\}$ and any $1\leq c_0,c_1,c_2,c_3\leq (\log X)^{C_2}$. Define, for any $\b{m}\in\N^4$,
    \[
    H''(X,\b{m}) = \mathop{\sum\sum\sum\sum}_{\substack{\b{n}\in\N^{4}\\ \|n_0n_1c_0,n_2n_3c_1\|\cdot\|m_0m_1c_2,m_2m_3c_3\|\leq X\\ \mathrm{gcd}(n_i,r_i)=1\;\forall 0\leq i\leq 3\\ n_i\equiv q_i\bmod{8}\;\forall 0\leq i\leq 3}}\frac{\psi_{Q_{02}m_0m_2}(n_0n_2)\psi_{Q_{13}m_1m_2}(n_1n_3)}{\tau(n_0)\tau(n_1)\tau(n_2)\tau(n_3)}.
    \]
    Then
    \begin{align*} \mathop{\sum\sum\sum\sum}_{\substack{\b{m}\in\N^{4},\|m_0,m_1,m_2,m_3\|\leq (\log X)^{C_3}\\ \mathrm{gcd}(m_0m_2,2Q_{02}r_0r_2)=\mathrm{gcd}(m_1m_3,Q_{13}r_1r_3)=1\\ Q_{02}m_0m_2\neq 1\;\text{and}\;Q_{13}m_1m_3\neq 1\\ \b{m}\equiv \Tilde{\b{q}}\bmod{8}}}& \hspace{-10pt}\frac{\mu^2(2m_0m_1m_2m_3)\lvert H''(X,\b{m})\rvert}{\tau(m_0)\tau(m_1)\tau(m_2)\tau(m_3)}\ll_{C_1,C_2,C_3,C_4}\frac{Q_{02}Q_{13}X^2}{c_0c_1c_2c_3(\log X)^{C_4}}.
    \end{align*}
    for any $C_4>0$ where the implied constant depends at most on the $C_i$.
\end{proposition}

\subsection{Small Conductor -- Non-Symmetric Hyperbola Method}\label{Lfunctionmethodsection}
As in the analogous part of \S \ref{hyperbolicanalysis1} the asymmetry of these sums leads to difficulty. In the previous case the lower bounds on some of the variables and averaging over the characters with the neutraliser large sieve led to saving over the desired bound. In this case we will likewise have to exploit the averaging over the conductor to obtain a valid bound, but our methods will differ as the convex factors $\frac{1}{n_0n_1}$ and $\frac{1}{\|n_0,n_1\|^2}$ switch roles from \S \ref{hyperbolicanalysis1}. The argument begins in a similar fashion to that of Lemma \ref{fixedconductorlemmaforvanishingmainterms}, but deviates in order to handle the need to average over our conductors. Assume that the character $\chi_{02}$ is non-principal with conductor $Q_{02}$ and consider
\begin{equation}\label{Generalnonsymmetricsumtype2}
A(X) = \mathop{\sum\sum\sum\sum}_{\substack{\|n_0n_1c_{01},n_2n_3c_{23}\|\cdot M \leq X \\ \mathrm{gcd}(n_i,r_i)=1\forall\;0\leq i\leq 3\\ n_i\equiv q_i\bmod{8}\forall\;0\leq i\leq 3}}\frac{\chi_{02}(n_0n_2)}{\tau(n_0)\tau(n_1)\tau(n_2)\tau(n_3)}.
\end{equation}
Define also the Dirichlet series
\[
\widetilde{L}_r(1,\chi) = \sum_{\substack{n=1 \\ \mathrm{gcd}(n,r)=1 \\ n\equiv q\bmod{8}}}^{\infty}\frac{\chi(n)}{\tau(n)}
\]
for any odd integer $r$, any $q\in(\Z/8\Z)^{*}$ and any non-principal Dirichlet character $\chi$. Our first step is to prove the following:

\begin{lemma}\label{Lfunctioncheckpoint}
    Let $X\geq 3$, $C_1,C_2,C_3,C_4,C_5>0$ and fix some $\b{q}\in(\Z/8\Z)^{*4}$. Let $2<Q_{02}\leq(\log X)^{C_1}$ and $1\leq r_0,r_1,r_2,r_3\leq(\log X)^{C_2}$ be odd integers such that $\mathrm{gcd}(Q_{02},2r_0r_2)=1$. Suppose $\chi_{02}$ is a non-principal character modulo $Q_{02}$. Then for any integers $1\leq c_{01},c_{23} \leq (\log X)^{C_3}$, $1\leq M\leq (\log X)^{C_4}$ we have
    \begin{align*}
    A(X) =& \frac{\mathfrak{S}_0(2r_1)\mathfrak{S}_0(2r_3)X^2}{16c_{01}c_{23}M^2\log X}\left(\mathop{\sum\sum}_{\substack{\chi,\chi'\bmod{8}}} \overline{\chi}(q_0)\overline{\chi'}(q_2) \widetilde{L}_{r_0}(1,\chi_{02}\chi)\widetilde{L}_{r_2}(1,\chi_{02}\chi')\right)\\ &+ O_{C_1,C_2,C_3,C_4,C_5}\left(\frac{X^2}{c_{01}c_{23}M^2(\log X)^{3/2}}\sum_{\substack{\chi\bmod{8}}}(\lvert \widetilde{L}_{r_0}(1,\chi_{02}\chi)\rvert+\lvert \widetilde{L}_{r_2}(1,\chi_{02}\chi)\rvert))\right) 
    \end{align*}
    where the implied constant depends at most on the $C_i$.
\end{lemma}

\begin{proof}
We write $A(X)$ as the product of two hyperbolic sums $H_{01}(X)H_{23}(X)$, where
\[
H_{01}(X) = \mathop{\sum\sum}_{\substack{n_0n_1 \leq X/c_{01}M \\ \mathrm{gcd}(n_i,r_i)=1\forall\;0\leq i\leq 1\\ n_i\equiv q_i\bmod{8}\forall\;0\leq i\leq 1}}\frac{\chi_{02}(n_0)}{\tau(n_0)\tau(n_1)}\;\;\textrm{and}\;\;
H_{23}(X) = \mathop{\sum\sum}_{\substack{n_2n_3 \leq X/c_{23}M \\ \mathrm{gcd}(n_i,r_i)=1\forall\;2\leq i\leq 3\\ n_i\equiv q_i\bmod{8}\forall\;2\leq i\leq 3}}\frac{\chi_{02}(n_2)}{\tau(n_2)\tau(n_3)}.
\]
We look at $H_{01}(X)$. Defining the parameter $Y=\exp((\log X)^{1/3})$, we use the standard hyperbola method we deduce that
\[
H_{01}(X) = H'_{01}(X) + H''_{01}(X) - H'''_{01}(X)
\]
where
\[
H'_{01}(X) = \hspace{-8pt}\mathop{\sum}_{\substack{n_0 \leq Y \\ \mathrm{gcd}(n_0,r_0)=1\\ n_0\equiv q_0\bmod{8}}}\frac{\chi_{02}(n_0)}{\tau(n_0)}\mathop{\sum}_{\substack{n_1 \leq X/n_0c_{01}M \\ \mathrm{gcd}(n_1,r_1)=1\\ n_1\equiv q_1\bmod{8}}}\frac{1}{\tau(n_1)},\;
H''_{01}(X) = \hspace{-8pt}\mathop{\sum}_{\substack{n_1 \leq X/c_{01}MY \\ \mathrm{gcd}(n_1,r_1)=1\\ n_1\equiv q_1\bmod{8}}}\frac{1}{\tau(n_1)}\mathop{\sum}_{\substack{n_0 \leq X/n_1c_{01}M \\ \mathrm{gcd}(n_0,r_0)=1\\ n_0\equiv q_0\bmod{8}}}\frac{\chi_{02}(n_0)}{\tau(n_0)},
\]
and
\[
H'''_{01}(X) = \mathop{\sum\sum}_{\substack{n_0 \leq Y,n_1\leq X/c_{01}MY \\ \mathrm{gcd}(n_i,r_i)=1\forall\;0\leq i\leq 1\\ n_i\equiv q_i\bmod{8}\forall\;0\leq i\leq 1}}\frac{\chi_{02}(n_0)}{\tau(n_0)\tau(n_1)}.
\]
Using Lemma \ref{Siegel--Walfisz1} for the sums over $n_0$ we see that
\begin{align*}
H'''_{01}(X)\ll_{C_5'}& \frac{\tau(r_0)Q_{02}X}{c_{01}M(\log X)^{1/2}(\log X)^{C_5'/3}},
\end{align*}
and
\begin{align*}
H''_{01}(X)\ll_{C_5'}& \sum_{n_1\leq X/c_{01}MY}\frac{\tau(r_0)Q_{02}X}{n_1c_{01}M(\log X/n_1c_{01}MY)^{(C_5'+1)/3}} \ll_{C_5'}\frac{\tau(r_0)Q_{02}X}{c_{01}M(\log X)^{C_5'/3}}
\end{align*}
where in each case we have used the bound $(\log X/c_{01}MY) = (\log X)(1+O((\log X)^{-2/3}))$ which follows from the fact that $\log c_{01}MY \ll (\log X)^{1/3}$. In the above bounds we can write $C_5' = 3C_5$ for some $C_5>0$ to obtain
\[
H''_{01}(X),H'''_{01}(X) \ll_{C_5} \frac{\tau(r_0)Q_{02}X}{c_{01}M(\log X)^{C_5}}.
\]
For $H'_{01}(X)$ we apply Lemma \ref{Siegel--Walfisz1} for non-principal characters with $C_5>0$ sufficiently large. This will give
\begin{align*}
H'_{01}(X) =& \hspace{-10pt}\mathop{\sum}_{\substack{n_0 \leq Y \\ \mathrm{gcd}(n_0,2r_0)=1\\ n_0\equiv q_0\bmod{8}}}\hspace{-10pt}\frac{\chi_{02}(n_0)}{\tau(n_0)} \hspace{-5pt}\left(\frac{\mathfrak{S}_0(2r_1)X}{n_0c_{01}M\sqrt{(\log X/n_0c_{01}M)}} \hspace{-3pt}+ O\hspace{-2pt}\left(\frac{X(\log\log 3r_1)^{3/2}}{n_0c_{01}M(\log X)^{3/2}}\right)\right),
\end{align*}
where we have used $\log X/c_{01} M \gg \log X$ coming from $c_{01}, M\leq (\log X)^{C_2}$. Using the bound
\[
\mathop{\sum}_{\substack{n_0\leq Y}}\frac{1}{n_0\tau(n_0)} \ll \sqrt{\log Y} \ll (\log X)^{1/6}
\]
we obtain
\begin{align*}
H'_{01}(X) =& \frac{\mathfrak{S}_0(2r_1)X}{c_{01}M}\hspace{-5pt}\mathop{\sum}_{\substack{n_0 \leq Y \\ \mathrm{gcd}(n_0,2r_0)=1\\ n_0\equiv q_0\bmod{8}}}\frac{\chi_{02}(n_0)}{n_0\tau(n_0)\sqrt{(\log X/n_0c_{01}M)}} + O_{C_3}\left(\frac{X(\log\log 3r_1)^{3/2}}{c_{01}M(\log X)^{4/3}}\right).
\end{align*}
For the front term we use the following:
\[
\frac{1}{\sqrt{(\log X/n_0c_{01}M)}} = \frac{1}{\sqrt{(\log X)}}\cdot\frac{1}{\left(1-\frac{\log n_0c_{01}M}{\log X}\right)^{1/2}} = \frac{1}{\sqrt{(\log X)}}\left(1 + O\left(\frac{1}{(\log X)^{2/3}}\right)\right).
\]
It follows that
\[
H'_{01}(X) = \frac{\mathfrak{S}_0(2r_1)X}{c_{01}M\sqrt{\log X}}\sum_{\substack{n_0\leq Y\\ \mathrm{gcd}(n_0,2r_0)=1\\ n_0\equiv q_0\bmod{8}}}\frac{\chi_{02}(n_0)}{n_0\tau(n_0)} + O_{C_3}\left(\frac{X}{c_{01}M\log X}\right).
\]
Next we detect the condition $8|n_0-q_0$ using Dirichlet characters. Thus the main term sum in $H'_{01}(X)$ becomes
\begin{equation}\label{cutoffdirichletseries}
\frac{1}{4}\sum_{\substack{\chi\bmod{8}}}\overline{\chi}(q_0)\sum_{\substack{n_0\leq Y\\ \mathrm{gcd}(n_0,r_0)=1 }}\frac{\chi_{02}\chi(n_0)}{n_0\tau(n_0)},
\end{equation}
where $\chi_{02}\chi(n_0) = \chi_{02}(n_0)\chi(n_0)$ is a non-principal character modulo $8Q_{02}$ and are non-principal since $Q_{02}$ is odd and $\chi_{02}$ non-principal. Using a similar argument to that seen in the proof of Lemma \ref{fixedconductorlemmaforvanishingmainterms}, $\widetilde{L}_{r_0}(1,\chi_{02}\chi)$ converges and 
\[
\widetilde{L}_{r_0}(1,\chi_{02}\chi) = P(1,r_0,\chi_{02}\chi')R(1,\chi_{02}\chi)L(1,\chi_{02}\chi)^{1/2}
\]
where $P(1,r_0,\chi_{02}\chi)$ and $R(1,\chi_{02}\chi)$ are
\[
\prod_{\substack{p\;\text{prime}\\p|r_0}}\left(1+\sum_{j=1}^{\infty}\frac{\chi_{02}\chi(p)^j}{(j+1)p^{j}}\right)^{-1}\;\;\text{and}\;\;\prod_{p\;\text{prime}}\left(1+\sum_{j=1}^{\infty}\frac{\chi_{02}\chi(p)^j}{(j+1)p^{j}}\right)\left(1-\frac{\chi_{02}\chi(p)}{p}\right)^{1/2}
\]
respectively, and
\[
L(1,\chi_{02}\chi) = \prod_{p\;\text{prime}}\left(1-\frac{\chi_{02}\chi(p)}{p}\right)^{-1}
\]
is the $L$-function for the character $\chi_{02}\chi$. 
Seeing this, we may extend the sum over $n_0$ in \eqref{cutoffdirichletseries} at the cost of an error term. This equation then becomes
\[
=\frac{1}{4}\sum_{\substack{\chi\bmod{8}}}\overline{\chi}(q_0)P(1,r_0,\chi_{02}\chi)R(1,\chi_{02}\chi)L(1,\chi_{02}\chi)^{1/2} + O_{C_5}\left(\frac{Q_{02}}{(\log X)^{C_5}}\right),
\]
where we have used partial summation and Lemma \ref{Siegel--Walfisz1} to bound the tail of this series. We therefore see that $H'_{01}(X)$ is equal to
\begin{align*}
\frac{\mathfrak{S}_0(2r_1)X}{4c_{01}M\sqrt{\log X}}\sum_{\substack{\chi\bmod{8}}}\overline{\chi}(q_0)P(1,r_0,\chi_{02}\chi)R(1,\chi_{02}\chi)L(1,\chi_{02}\chi)^{1/2}+ O_{C_3}\left(\frac{X}{c_{01}M\log X}\right).
\end{align*}
Putting this together with $H''_{01}(X)$ and $H'''_{01}(X)$ we see that $H_{01}(X)$ is then
\begin{align*}
\frac{\mathfrak{S}_0(2r_1)X}{4c_{01}M\sqrt{\log X}}\sum_{\substack{\chi\bmod{8}}}\overline{\chi}(q_0)P(1,r_0,\chi_{02}\chi)R(1,\chi_{02}\chi)L(1,\chi_{02}\chi)^{1/2}+ O_{C_3}\left(\frac{X}{c_{01}M\log X}\right).
\end{align*}
Similarly we may obtain that $H_{23}(X)$ is
\begin{align*}
=\frac{\mathfrak{S}_0(2r_3)X}{4c_{23}M\sqrt{\log X}}\sum_{\substack{\chi' \; \text{char.}\\ \bmod{8}}}\overline{\chi'}(q_0)P(1,r_3,\chi_{02}\chi')R(1,\chi_{02}\chi')L(1,\chi_{02}\chi')^{1/2}+ O_{C_3}\left(\frac{X}{c_{23}M(\log X)}\right).
\end{align*}
Multiplying these together we obtain the result.
\end{proof}

In order to take the desired averages over the characters, we will use the fact that $\widetilde{L}_{r}(1,\chi)$ looks roughly like $L(1,\chi)^{1/2}$. In fact, by noting the bounds
\[
P(1,r,\chi) \ll \tau(r),\;
R(1,\chi) \ll 1
\]
for any non-principal character $\chi$ and absolute implied constants, we observe
\begin{equation}\label{DirichletFunctionDecomposition}
 \widetilde{L}_{r}(1,\chi)^2 \ll \tau(r)^2L(1,\chi)
\end{equation}
since $L(1,\chi)>0$ for real non-principal characters $\chi$.\\

Recall that we use the notation $\psi_{m}(\cdot)$ for an odd integer $m$ to denote generically the Jacobi symbol $\left(\frac{\cdot}{m}\right)$ or $\left(\frac{m}{\cdot}\right)$. We may use quadratic reciprocity to interchange between the two if necessary (in the following proof, the characters modulo $8$ and square-free functions ensure that the variables are odd). The following is a consequence of Lemma \ref{Lfunctionweirdaverage}.

\begin{corollary}\label{Lfunctionweirdaveragecor}
    Let $X\geq 3$ and $C>0$. Fix some real number $0<c\leq 1$. Suppose that $\chi$ is a character modulo $8$. Then
    \[
    \mathop{\sum\sum}_{\substack{\|m_0,cm_1\|\leq X\\m_0m_1\neq 1}}\frac{\mu^2(2m_0m_1)}{\tau(m_0)\tau(m_1)}L(1,\chi\cdot\psi_{m_0m_1}) \ll \frac{X^2}{c\sqrt{\log X}},
    \]
    where the implied constant is absolute.
\end{corollary}

\begin{proof}
    When $m_1=1$ then we use Lemma \ref{Lfunctionweirdaverage} with $f=\frac{1}{\tau}$, $q=2$, for the sum over $m_0$, since $m_0m_1\neq 1$. When $m_1>1$ then we may use Lemma \ref{Lfunctionweirdaverage} with $f=\frac{1}{\tau}$, $q=2$, to bound the sum over $m_1$. In this case:
    \begin{align*}
    \mathop{\sum\sum}_{\substack{\|m_0,cm_1\|\leq X\\m_1> 1}}\frac{\mu^2(2m_0m_1)}{\tau(m_0)\tau(m_1)}L(1,\chi\cdot\psi_{m_0m_1}) &\ll \mathop{\sum}_{\substack{m_0\leq X}}\frac{X}{c\sqrt{\log X/c}}\ll \frac{X^2}{c\sqrt{\log X}}.
    \end{align*}
\end{proof}

We may now prove the main result of this section:
\begin{proposition}\label{Asymmetrictypeaverage2}
    Let $X\geq 3$, $C_1,C_2>0$, fix $\b{q}\in(\Z/8\Z)^{*4}$ and $\Tilde{\b{q}}\in(\Z/8\Z)^{*2}$ and let $\Tilde{\b{r}}\in\N^2$ be a vector of odd integers. Fix odd integers $1\leq r_0,r_1,r_2,r_3,\Tilde{r}_0,\Tilde{r}_1\leq (\log X)^{C_1}$ and fix $1\leq c_{01},c_{23},\Tilde{c}_0,\Tilde{c}_1\leq (\log X)^{C_2}$. Then for any $\b{m}\in\N^{2}$ we define
    \[
    T(X,\b{m}) = \mathop{\sum\sum\sum\sum}_{\substack{\|n_0n_1c_{01},n_2n_3c_{23}\|\cdot \|m_0\Tilde{c}_0,m_1\Tilde{c}_1\| \leq X \\ \mathrm{gcd}(n_i,2r_i)=1\forall\;0\leq i\leq 3\\ n_i\equiv q_i\bmod{8}\forall\;0\leq i\leq 3}}\frac{\psi_{m_0m_1}(n_0n_2)}{\tau(n_0)\tau(n_1)\tau(n_2)\tau(n_3)}.
    \]
    Then for any $C_3>0$,
    \[
    \mathop{\sum\sum}_{\substack{\|m_0,m_1\|\leq (\log X)^{C_3}\\ m_i\equiv \Tilde{q}_i\bmod{8}\;\forall 0\leq i\leq 1\\ \mathrm{gcd}(m_i,\Tilde{r}_i)=1\;\forall 0\leq i\leq 1\\ m_0m_1\neq 1}}\frac{\mu^2(m_0m_1)}{\tau(m_1)\tau(m_2)} |T(X,\b{m})| \ll_{C_1,C_2,C_3} \frac{\tau(r_0)\tau(r_2)X^2(\log\log X)^{1/2}}{c_{01}c_{23}\Tilde{c}_0\Tilde{c}_1(\log X)}.
    \]
    where the implied constant depends at most on the $C_i$.
\end{proposition}

\begin{proof}
    Using Lemma \ref{Lfunctioncheckpoint} on the $T(X,\b{m})$ and the triangle inequality we see that the sum over $\b{m}$ in the propostion is 
    bounded by
    \begin{align*}
    \frac{\mathfrak{S}_0(2r_1)\mathfrak{S}_0(2r_3)X^2}{16c_{01}c_{23}(\log X)}\hspace{-2pt}\mathop{\sum\sum}_{\substack{\chi,\chi'\bmod{8}}} M(X,\chi,\chi') + O\left(\frac{X^2}{c_{01}c_{23}(\log X)^{3/2}}\hspace{-5pt}\mathop{\sum}_{\substack{\chi\bmod{8}}}\hspace{-5pt}E(X,\chi)\right),
    \end{align*}
    where
    \[
    M(X,\chi,\chi') = \hspace{-5pt}\mathop{\sum\sum}_{\substack{\|m_0,m_1\|\leq (\log X)^{C_3}\\ m_i\equiv \Tilde{q}_i\bmod{8}\;\forall 0\leq i\leq 1\\ \mathrm{gcd}(m_i,\Tilde{r})=1\;\forall 0\leq i\leq 1\\ m_0m_1\neq 1}}\frac{\mu^2(m_0m_1)}{\tau(m_1)\tau(m_2)\|m_0\Tilde{c}_0,m_1\Tilde{c}_1\|^2}|\widetilde{L}_{r_0}(1,\psi_{m_0m_1}\chi)||\widetilde{L}_{r_2}(1,\psi_{m_0m_1}\chi')|
    \]
    and
    \[
    E(X,\chi) = \sum_{j=0,2}\mathop{\sum\sum}_{\substack{\|m_0,m_1\|\leq (\log X)^{C_3}\\ m_0m_1\neq 1}}\frac{\mu^2(m_0m_1)}{\tau(m_1)\tau(m_2)\|m_0\Tilde{c}_0,m_1\Tilde{c}_1\|^2}\left|\widetilde{L}_{r_j}(1,\psi_{m_0m_1}\chi)\right|.
    \]
    We first bound $M(X,\chi,\chi')$. We may write this sum in the form
    \[
    M(X,\chi,\chi') = \mathop{\sum\sum}_{\substack{\|m_0,m_1\|\leq (\log X)^{C_3}}} a_{(m_0,m_1)} b_{(m_0,m_1)}c_{(m_0,m_1)}
    \]
    where $b_{(m_0,m_1)} = |\widetilde{L}_{r_0}(1,\psi_{m_0m_1}\chi)|$, $c_{(m_0,m_1)} = |\widetilde{L}_{r_2}(1,\psi_{m_0m_1}\chi')|$ and $a_{(m_0,m_1)}$ represents the remaining summands and conditions. We may therefore re-index this sum as a sum over a single variable,
    \[
    M(X,\chi,\chi') = \sum_{l\leq (\log X)^{2C_3}} \tilde{a}_l \tilde{b}_l \tilde{c}_l.
    \]
    Using Cauchy's inequality, and then returning to the original double indexing, we obtain
    \[
    M(X,\chi,\chi') = \left(\mathop{\sum\sum}_{\substack{\|m_0,m_1\|\leq (\log X)^{C_3}}} a_{(m_0,m_1)} b_{(m_0,m_1)}^2\right)^{1/2}\left(\mathop{\sum\sum}_{\substack{\|m_0,m_1\|\leq (\log X)^{C_3}}} a_{(m_0,m_1)}c_{(m_0,m_1)}^2\right)^{1/2}.
    \]
    In other words, we have now obtained $M(X,\chi,\chi')\ll R_{r_0}(X,\chi)^{1/2}R_{r_2}(X,\chi')^{1/2}$ where
    \begin{align*}
    R_r(X,\chi) = \mathop{\sum\sum}_{\substack{\|m_0,m_1\|\leq (\log X)^{C_3}\\ m_0m_1\neq 1}}\frac{\mu^2(2m_0m_1)}{\tau(m_1)\tau(m_2)\|m_0\Tilde{c}_0,m_1\Tilde{c}_1\|^2}\lvert \widetilde{L}_r(1,\psi_{m_0m_1}\chi)\rvert^2.
    \end{align*}
    By \eqref{DirichletFunctionDecomposition} we have
    \[
    R_r(X,\chi) \ll \tau(r)^2\mathop{\sum\sum}_{\substack{\|m_0,m_1\|\leq (\log X)^{C_3}\\ m_0m_1\neq 1}}\frac{\mu^2(2m_0m_1)}{\tau(m_1)\tau(m_2)\|m_0\Tilde{c}_0,m_1\Tilde{c}_1\|^2} L(1,\psi_{m_0m_1}\chi).
    \]
    Writing
    \[
    a(M) = \mathop{\sum\sum}_{\substack{\b{m}\in\N^2, \|m_0\Tilde{c}_0,m_1\Tilde{c}_1\|=M\\ m_0m_1\neq 1}}\frac{\mu^2(2m_0m_1)}{\tau(m_1)\tau(m_2)} L(1,\psi_{m_0m_1}\chi)
    \]
    we obtain
    \begin{equation}\label{checkpoint1forsection4asymmetric}
    R_r(X,\chi) \ll \tau(r)^2\sum_{2\leq M\leq \|\Tilde{c}_0,\Tilde{c}_1\|(\log X)^{C_3}}\frac{a(M)}{M^2}.
    \end{equation}
    By partial summation the sum on the right hand side of \eqref{checkpoint1forsection4asymmetric} is then 
    \begin{align*}
     \frac{1}{\|\Tilde{c}_0,\Tilde{c}_1\|^2(\log X)^{2C_3}}\sum_{2\leq M\leq \|\Tilde{c}_0,\Tilde{c}_1\|(\log X)^{C_3}}\hspace{-25pt}a(M)+ 2\int_{2}^{\|\Tilde{c}_0,\Tilde{c}_1\|(\log X)^{C_3}}\frac{\sum_{2\leq M\leq t}a(M)}{t^3}dt. 
    \end{align*}
    Using the fact that $a(M) = 0$ unless $M\geq \|\Tilde{c}_0,\Tilde{c}_1\|$ and the change of variables $t=\|\Tilde{c}_0,\Tilde{c}_1\|u$ the integral becomes
    \[
    \int_{2}^{\|\Tilde{c}_0,\Tilde{c}_1\|(\log X)^{C_3}}\frac{\sum_{2\leq M\leq t}a(M)}{t^3}dt = \int_{1}^{(\log X)^{C_3}}\frac{\sum_{2\leq M\leq \|\Tilde{c}_0,\Tilde{c}_1\|u}a(M)}{\|\Tilde{c}_0,\Tilde{c}_1\|^2u^3}du.
    \]
    Thus
    \begin{align}\label{checkpoint3forsection4asymmetric}
    \frac{R_r(X,\chi)}{\tau(r)^2} \ll \frac{1}{\|\Tilde{c}_0,\Tilde{c}_1\|^2(\log X)^{2C_3}}\sum_{2\leq M\leq \|\Tilde{c}_0,\Tilde{c}_1\|(\log X)^{C_3}}\hspace{-25pt}a(M) +  \int_{1}^{(\log X)^{C_3}}\frac{\sum_{2\leq M\leq \|\Tilde{c}_0,\Tilde{c}_1\|u}a(M)}{\|\Tilde{c}_0,\Tilde{c}_1\|^2u^3}dt. 
    \end{align}
    Now, upon unwrapping $a(M)$, we may see that,
    \begin{align*}
    \sum_{2\leq M\leq \|\Tilde{c}_0,\Tilde{c}_1\|Y} a(M) &=  \mathop{\sum\sum}_{\substack{ \|m_0\frac{\Tilde{c}_0}{\|\Tilde{c}_0,\Tilde{c}_1\|},m_1\frac{\Tilde{c}_1}{\|\Tilde{c}_0,\Tilde{c}_1\|}\|\leq Y\\ m_0m_1\neq 1}}\frac{\mu^2(2m_0m_1)}{\tau(m_1)\tau(m_2)} L(1,\psi_{m_0m_1}\chi).
    \end{align*}
    By Corollary \ref{Lfunctionweirdaveragecor} with $c=\frac{\min(\Tilde{c}_0,\Tilde{c}_1)}{\|\Tilde{c}_0,\Tilde{c}_1\|}$ we get
    \begin{align*}
    \frac{R_r(X,\chi)}{\tau(r)^2} &\ll_{C_2,C_3} \frac{1}{\min(\Tilde{c}_0,\Tilde{c}_1)\|\Tilde{c}_0,\Tilde{c}_1\|\sqrt{\log\log X}} + \int_{1}^{(\log X)^{C_3}}\hspace{-22pt}\frac{1}{\min(\Tilde{c}_0,\Tilde{c}_1)\|\Tilde{c}_0,\Tilde{c}_1\| u\sqrt{\log u}}dt.
    \end{align*}
    This is $O(\sqrt{\log\log X}/\Tilde{c}_0\Tilde{c}_1)$, hence
    \[
    M(X,\chi,\chi') \ll_{C_2,C_3} \frac{\tau(r_0)\tau(r_2)\sqrt{\log\log X}}{\Tilde{c}_0\Tilde{c}_1}.
    \]
    To deal with $E(X,\chi)$ we treat the sum over each $\widetilde{L}_{r_j}(1,\psi_{m_0m_1}\chi)$ separately. Calling each one $E_j(X,\chi)$ we once more use Cauchy's inequality as before to get $E_j(X,\chi) \ll (\mathcal{E}_j(X,\chi)\mathcal{E}_j'(X))^{1/2}$
    where
    \[
    \mathcal{E}_j(X,\chi) = \tau(r_j)^2\mathop{\sum\sum}_{\substack{\|m_0,m_1\|\leq (\log X)^{C_3}\\ m_0m_1\neq 1}}\frac{\mu^2(m_0m_1)}{\tau(m_1)\tau(m_2)\|m_0\Tilde{c}_0,m_1\Tilde{c}_1\|^2} L(1,\psi_{m_0m_1}\chi)
    \]
    and
    \[
    \mathcal{E}_j'(X) = \mathop{\sum\sum}_{\substack{\|m_0,m_1\|\leq (\log X)^{C_3}\\ m_0m_1\neq 1}}\frac{\mu^2(m_0m_1)}{\tau(m_1)\tau(m_2)\|m_0\Tilde{c}_0,m_1\Tilde{c}_1\|^2}.
    \]
    Using similar techniques to those used to bound $M(X,\chi,\chi')$ to bound $\mathcal{E}_j(X,\chi)$ and Lemma \ref{maintermlemma3} to bound $\mathcal{E}'_j(X)$ we obtain
    \[
    \frac{\mathcal{E}_j(X,\chi)}{\tau(r_0)^2},\mathcal{E}'_j(X) \ll \frac{(\log\log X)^{1/2}}{(\Tilde{c}_0\tilde{c}_1)}
    \]
    from which it follows that
    \[
    E(X,\chi') \ll_{C_2,C_3} \frac{\tau(r_0)\tau(r_1)(\log\log X)^{1/2}}{\Tilde{c}_0\Tilde{c}_1}.
    \]
    Finally, we inject the bounds for $M(X,\chi,\chi')$ and $E_j(X,\chi)$ into our overall expression, summing over finitely many characters $\chi$ and $\chi'$ modulo $8$ and noting that $\mathfrak{S}_0(r_i)\ll 1$ for all integers, we conclude the proof.    
\end{proof}

\end{document}